\documentclass[11pt,a4paper]{article}
\usepackage{fancyhdr}
\usepackage{amsmath}
\usepackage{amsthm}
\usepackage{amssymb}
\usepackage{wasysym}
\usepackage{paralist}
\usepackage[all]{xy}
\usepackage{yhmath}

\usepackage{young}
\usepackage[vcentermath]{youngtab} 

\usepackage{titling}

\newcommand{\nc}{\newcommand}
 
 \newcommand{\pol}{{\rm pol\,}}
\nc{\sgn}{{\rm sgn}}
\newcommand{\barA}{{\bar A}}
\newcommand{\barc}{{\bar c}}
\newcommand{\barY}{{\bar Y}}
\newcommand{\cf}{{\rm cf}}
\newcommand{\id}{{\rm id}}
\newcommand{\hate}{{\hat e}}
\newcommand{\Comod}{{\rm Comod}}
\newcommand{\comod}{{\rm comod}}
\newcommand{\Mod}{{\rm Mod}}
 \renewcommand{\l}{{\rm len}}
\newcommand{\K}{\mathbb{K}}
\newcommand{\Z}{\mathcal{Z}}
\newcommand{\zz}{\mathbb{Z}}
\newcommand{\gtdom}{\triangleright}
\newcommand{\ledom}{\trianglelefteq}
\newcommand{\kerp}{\mathcal{K}}
\newcommand{\hecke}{\mathcal{H}}
\newcommand{\tab}{\mathcal{T}}
\newcommand{\bmcom}{\mathcal{C}}
\newcommand{\lng}{\mathrm{lng}}
\newcommand{\ideal}{\mathcal{J}}
\DeclareMathOperator{\Ext}{Ext}
\DeclareMathOperator{\Hom}{Hom}
\DeclareMathOperator{\Tor}{Tor}
\renewcommand{\vert}{{\,|\,}}
\renewcommand{\P}{{\mathcal P}}
\newcommand{\que}{{\mathbb Q}}
\newcommand{\rat}{{\mathbb Q}}

\newcommand{\q}{}
\newcommand{\de}{\delta}
\newcommand{\ep}{{\epsilon}}
\renewcommand{\mod}{{\rm mod}}
\newcommand{\bs}{\bigskip}
\renewcommand{\vert}{\,|\,}
\newcommand{\tbw}{\textstyle\bigwedge}
\newcommand{\zed}{{\mathbb Z}}
\newcommand{\G}{{\mathcal G}}
\renewcommand{\H}{{\mathcal H}}
\newcommand{\F}{{\mathcal  F}}
\nc{\Q}{{\mathcal Q}}

\renewcommand{\Hom}{{\rm Hom}}
\newcommand{\resp}{{\rm resp.\,}}
\newcommand{\ind}{{\rm ind}}

\newtheorem{definition}{Definition}[section]
\newtheorem{proposition}[definition]{Proposition}
\newtheorem{theorem}[definition]{Theorem}
\newtheorem{lemma}[definition]{Lemma}
\newtheorem{corollary}[definition]{Corollary}

\newtheorem{remark}[definition]{Remark}

\newcommand{\si}[1]{\small{\emph{#1}}}

\preauthor{}
\postauthor{}
\DeclareRobustCommand{\authorthing}{
\begin{center}
	\begin{tabular}{lll}
		Stephen Donkin\thanks{The first author wishes to thank CMUC for hospitality
in May~2012, when the major part of the research for this paper was done,  and
in  April~2013, when the final version was produced. } & Ana Paula Santana\thanks{This work was partially supported by the Centro de Matem\'atica da
Universidade de Coimbra (CMUC), funded by the European Regional
Development Fund through the program COMPETE and by the Portuguese
Government through the FCT - Funda\c{c}\~ao para a Ci\^encia e a Tecnologia
under the project PEst-C/MAT/UI0324/2011.
} &
		Ivan Yudin\thanksmark{2}\thanksgap{0.3em}\thanks{The third author's work is supported by the FCT Grant
SFRH/BPD/31788/2006.}\\	
\si{Department of}
& \si{CMUC, Department of} & \si{CMUC, Department of}\\
\si{Mathematics}&
\si{Mathematics}&
\si{Mathematics}\\
\si{University of York}
&\si{University of Coimbra}
&\si{University of Coimbra}\\
\si{Heslington, York}
&\si{Coimbra} & \si{Coimbra}\\
\si{UK} & \si{Portugal} &\si{Portugal}\\
\si{stephen.donkin@york.ac.uk} & \si{aps@mat.uc.pt} & \si{yudin@mat.uc.pt}
	\end{tabular}
\end{center}
}

\author{\authorthing}

\title{Homological properties of quantised Borel-Schur algebras and resolutions
of quantised Weyl modules}

\begin{document}
\maketitle
\abstract{We continue the development of the homological theory of quantum
general linear groups previously considered by the first author. The development
is used to transfer information to the representation theory of quantised Schur
algebras. The acyclicity of induction from some rank-one modules for quantised
Borel-Schur subalgebras is deduced. This is used to prove the exactness of the
complexes recently constructed by Boltje and Maisch, giving resolutions of the
co-Specht modules for Hecke algebras.}

 \section{Introduction}

In~\cite{advances} the last two authors constructed characteristic free
projective resolutions  of the Weyl modules for the
classical Schur algebra.
Then, using the Schur functor, obtained resolutions by permutation modules of
the co-Specht modules for the symmetric group. This last result allowed them to
prove Conjecture 3.4 of Boltje and Hartman~\cite{boltjehartmann}. The key
ingredients of~\cite{advances} are the use of the normalised bar resolution in
the context of Borel-Schur algebras and Woodcock's Theorem~\cite{woodcock},
which reduces the construction of projective resolutions for Weyl modules to the
construction of projective resolutions for rank-one modules for the Borel-Schur
algebra.
The original motivation of the present paper was to extend the results
of~\cite{advances} to the context  of quantised Schur algebras and Hecke
algebras. This is easily achieved once one has a quantised version of the
generalization of Woodcock's theorem given in~\cite{advances}. 

Fix positive integers $n$ and $r$,  a commutative ring $R$ and an invertible
element~$q$ in $R$. Consider the quantised Schur algebra $S_{R,q}\left( n,r
\right)$ and the quantised positive Borel-Schur algebra $S^+_{R,q}\left( n,r \right)$. 
For each partition $\lambda=(\lambda_1, \dots, \lambda_n)$ of~$r$ there is a
rank-one module $R_{\lambda}$ for $S^+_{R,q}\left( n,r \right)$. The induced module 
\begin{equation*}
	W^{R,q}_{\lambda} := S_{R,q}\left( n,r \right)\otimes_{S^+_{R,q}\left(
	n,r \right)}
	R_{\lambda}
\end{equation*}
	is the Weyl module associated with $\lambda$.
Following~\cite{advances}, we work in the category of  $S^+_{R,q}\left( n,r
\right)$-modules and use the normalised bar resolution to construct a projective
resolution of $R_{\lambda}$. Next we apply the induction functor $S_{R,q}\left(
n,r \right)\otimes_{S^+_{R,q}\left( n,r \right)}-$\ \ \ to this resolution
 and
 obtain a complex $B^{R,q}_{*,\lambda}$ of finite length 
\begin{equation*}
 \dots \to B^{R,q}_{1,\lambda} \to
 B^{R,q}_{0,\lambda} \to W^{R,q}_\lambda \to 0,
\end{equation*}
where each $B_{k,\lambda}^{R,q}$ is a projective $S_{R,q}(n,r)$-module.

To show that this complex is exact we use Theorem~\ref{mainfirst}, which is the quantised version of  Woodcock's
Theorem. So
$B^{R,q}_{*,\lambda
}$ is a projective resolution of the quantised Weyl  
module $W_\lambda^{R,q}$ and it is simple to see that this resolution is
universal, that is 
\begin{equation*}
	B_{*,\lambda}^{R,q} \cong B_{*, \lambda}^{\Z, t}
	\otimes_{\Z} R, 
\end{equation*}
where $\Z = \zz\left[ t,t^{-1} \right]$ is the universal quantization ring.

Write $\hecke_{R,q}$ for the Hecke algebra over $R$ associated with the
symmetric group $\Sigma_r$. In~\cite{boltjemaish}, Boltje and Maisch
constructed,
for each composition $\lambda = (\lambda_1, \dots, \lambda_n)$ of $r$, a
complex
$\widetilde\bmcom^\lambda_*$ of left $\hecke_{R,q}$-modules and proved that
it
is exact in degrees $0$ and $-1$. Specializing to $q=1$,
$\widetilde\bmcom^\lambda_*$ coincides with the complex  
constructed in~\cite{boltjehartmann}. Suppose that $\lambda$ is a partition
of
$r$. Then the last module in $\widetilde\bmcom^\lambda_*$ is the dual of the
Specht module $S^\lambda$ over $\hecke_{R,q}$. 
It was proved in~\cite{advances} that in this situation upon specializing to
$q=1$ the resulting complex is exact. It is natural to conjecture that the
same
should be true for an arbitrary $q$. 

Returning to our setting, we choose $n\ge r$ and fix $\lambda$ a partition
of
$r$ into at most $n$ parts. We apply the Schur functor
\begin{align*}
	F\colon S_{R,q}(n,r)\mbox{-mod} \to \hecke_{R,q}\mbox{-mod}
\end{align*}
to our resolution $B^{R,q}_{*,\lambda}$ and obtain an exact complex $F\left(
B^{R,q}_{*,\lambda} \right)$ which we prove to be isomorphic to
$\widetilde\bmcom^\lambda_*$. This proves the exactness of
$\widetilde\bmcom^\lambda_*$.

\q We approach the  quantisation of Woodcock's Theorem, as described above,
    via the representation theory of the quantum general linear group~$G(n)$ of
    degree $n$,  introduced in \cite{SD25}.  In fact we take this  opportunity
    to develop the homological theory   previously considered in \cite{SD37}
    and \cite{SD42}.  The focus here is on a comparison between  the homological
    algebra in the category of polynomial modules and in  the full category of
    modules for the quantum group.  We work over an arbitrary field $\K$ and
    non-zero parameter $q\in \K$.

    \q Let $B(n)$ be the negative Borel (quantum) subgroup of $G$. We prove in
    particular that the derived functors of induction $R^i\ind_{B(n)}^{G(n)}$
    take polynomial modules to polynomial modules, Corollary 7.7.  Furthermore we show that if
    $V$ is a homogeneous polynomial $B(n)$-module of degree $r$  then
    $R^i\ind_{B(n)}^{G(n)}V=0$ for all $i>r$, Lemma 6.2. In general the tensor product is
    not commutative in the category of modules for a quantum group. However, we
    show that if  $L$ and $M$ are $B$-modules and $L$ is one dimensional then
    the $B$-module $L\otimes M$ and $M\otimes L$ are isomorphic, Proposition~7.1. Using this
    property and a Koszul resolution we show that the polynomial part of the
    coordinate algebra of  $B(n)$ is acyclic for the induction functor. This
    leads to the fact that the derived functors of induction applied to  a
    polynomial $B(n)$-module are the same whether computed in the polynomial
    category or the full module category, Theorem 7.5.   Kempf's Vanishing
    Theorem for representations of quantum groups,  when expressed in the
    polynomial
    category, is essentially the quantised version of Woodcock's Theorem, over a
    field.  Some further work is needed to expressed this in terms of the
    acyclicity theorem for induction over Schur algebras, over an arbitrary
    coefficient ring, mentioned above, Theorem 8.4.

    \q Though not needed for the application to resolutions we also take the
    opportunity to give the generalisation to the quantum Borel subgroup
    $B(n)$ of another theorem of Woodcock, \cite{WoodcockBielefeld}, Theorem 7
    and \cite{WoodcockComm} (see also \cite{WoodcockMasterpiece} for related
    material obtained by working with global bases). This theorem asserts that
    the extension groups between polynomial $B(n)$-modules of the same degree
    whether calculated in the polynomial category or the full $B(n)$-module
    category are the same, Theorem 5.2. We approach the quantised version by considering the  derived
    functors of the functor $\pol$, which takes a $B(n)$-module to its largest
    polynomial submodule. Though in detail it looks quite different it is in
    spirit rather close to the approach of  \cite{WoodcockBielefeld}, and we
    gratefully acknowledge the influence of this unpublished work.

The organization of the present paper is as follows. We first study the
homological results for quantum $G(n)$ and its negative Borel subgroup. Then we
use this to obtain the quantised version  of Woodcock's Theorem, Theorem~\ref{mainfirst}. 
In the last part of the paper we construct universal projective resolutions for
quantised Weyl modules. Using these resolutions, we prove the exactness of Boltje and Maisch
complexes for dominant weights.
\section{Restriction and induction of comodules}

\rm
\bs

\q We fix a field $\K$. For a vector space $V$  over $\K$ we write $V^*$ for the linear dual $\Hom_\K(V,\K)$ and if $W$ is also a  vector space over $\K$ we write simply $V\otimes W$ for the tensor product $V\otimes_\K W$.  We write $\id_X$ for the identity map on a set $X$.

\q For a coalgebra $A=(A,\de_A,\ep_A)$ over $\K$ we write $\Comod(A)$ for the
category of right $A$-comodules and write $\comod(A)$ for the category of finite dimensional right $A$-comodules.   We recall for future use the definition of the coefficient space of an $A$-comdodule.   Let $V=(V,\tau)$ be a right $A$-comodule and let $\{v_i:i\in I\}$,  be a $\K$-basis of $V$. The coefficient space $\cf(V)$ is the $\K$-span of the elements $f_{ij}\in A$ defined by the equations
$$\tau(v_i)=\sum_{j\in I} v_j\otimes f_{ji}$$
for $i\in I$. (This space is independent  of the choice of basis.  For further properties see \cite{LFR}.)

\q Let $B=(B,\de_B,\ep_B)$ also be a coalgebra and suppose $\phi\colon A\to B$ is a coalgebra map.
 Recall that for $V=(V,\tau)\in \Comod(A)$ we have  
 \begin{equation*} 
  \phi_0(V)=(V,(\id_V\otimes \phi)\circ\tau)\in \Comod(B).
   \end{equation*}
	 If $f\colon V\to V'$ is a morphism of right comodules then  the same
	 map $f\colon V\to V'$ is also a morphism of $B$-comodules. In this way we have an exact functor $\phi_0\colon \Comod(A)\to \Comod(B)$, with $\phi_0(f)=f$, for  $f$ a morphism of $A$-comodules. We call $\phi_0$ the $\phi$-restriction (or just restriction) functor.

\q More interestingly perhaps, we have  the $\phi$-induction functor \\
$\phi^0\colon \Comod(B)\to \Comod(A)$. This is described in  \cite{SD1}, Section 3,   and we briefly recall  the construction and some properties.  
If $X$ is a $\K$-vector space (possibly with extra structure) we write $|X|\otimes A$ for the vector space $X\otimes A$ regarded as an $A$-comodule with structure map $\id_X\otimes \de_A$.
Let  $(W,\mu)\in \Comod(B)$.   The  set  of all $s\in W\otimes A$ such that 
$$(\mu\otimes \id_A)(s)=(\id_W\otimes (\phi\otimes \id_A)\circ \de_A)(s)\in W\otimes B\otimes A$$
 is an $A$-subcomodule of $|W|\otimes A$, which we denote $\phi^0(W)$.  If $f\colon W\to W'$ is a morphism of $B$-comodules  then the map $f\otimes \id_A$ restricts to an  $A$-comodule map  $\phi^0(f)\colon \phi^0(W)\to \phi^0(W')$. In this way we obtain a left exact functor $\phi^0\colon \Comod(B)\to \Comod(A)$.  Let  $V=(V,\lambda)\in \Comod(A)$ and $W=(W,\mu)\in \Comod(B)$. 
 We have a natural isomorphism $ \Hom_B(\phi_0(V),W))\to \Hom_A(V,\phi^0(W))$, taking $\alpha\in \Hom_B(\phi_0(V),W))$ to $\tilde\alpha=(\alpha\otimes \id_A)\circ \lambda$.
 
\q Suppose now that $A$ is finite dimensional. We consider the dual algebra $S=A^*=\Hom_\K(A,\K)$.  Given a right $A$-comodule $V$ with structure map $\tau\colon V\to V\otimes A$ we may also regard $V$ as a left $S$-module with action 
$\alpha v = (\id_V\otimes \alpha) \tau(v)$.  If $\theta\colon  V\to V'$ is a morphism
of right $A$-comodules then, regarding $V$ and $V'$ as left $S$-modules,
$\theta\colon V\to V'$ is also a morphism in the category of left $S$-modules. In this
way we have an equivalence  between the categories of  finite dimensional right
$A$-comodules and  of finite dimensional left $S$-modules.  For finite dimensional right $A$-comodules $V,V'$ this equivalence of categories induces a  $\K$-linear isomorphism $\Ext^i_A(V,V')\to \Ext^i_S(V,V')$ in each degree $i$.

\q If $S$ is a $\K$-algebra and $V$ is a left (resp. right) $A$-module then the linear dual $V^*$ is naturally a right (\resp  left) $S$-module.   Now suppose $\phi\colon A\to B$ is a morphism of finite dimensional $\K$-coalgebras and let $T=B^*$.  The linear dual $\phi^*\colon T\to S$ is a $\K$-algebra map.   Now $A$ is naturally an $(S,S)$-bimodule with left action $\alpha a=(\id_A\otimes \alpha)\de_A(a)$ and right action $a \beta=(\beta\otimes \id_A)\de(a)$, for $a\in A$, $\alpha,\beta\in S$. We view an $S$-module also as a $T$-module via $\phi^*$.

\q We have the natural linear isomorphism  $\eta\colon  V\otimes A\to (V^*\otimes A^*)^*$.   The tensor product $V^*\otimes_T A^*$ is a quotient of $(V^*\otimes A^*)$ and we thus identify $(V^*\otimes_T A^*)^*$ with a subspace of  $(V^*\otimes A^*)^*$. From the definitions one checks  
 that an element $y$ of $V\otimes A$ lies in $\phi^0(V)$ if and only if $\eta(y)$ lies in  $(V^*\otimes_T  A^*)^*$. The map $\eta$ restricts to an isomorphism of left $A$-modules 
 \begin{equation*}
  \phi^0V\to (V^*\otimes_T A^*)^*.
   \end{equation*}
	 It follows that the derived functors of $\phi^0$ are given as follows.
\bs

\begin{proposition}Let $\phi\colon A\to B$ be  a morphism of finite dimensional
	coalgebras over $\K$. Then for $V\in \comod(B)$ we have
$$R^i\phi^0V= (\Tor_i^{B^*}(V^*,A^*))^*$$
for $i\geq 0$.
\end{proposition}

\section{The quantum polynomial algebra in $n^2$ variables}

\q We shall work with the quantum general linear groups  defined in  \cite{SD25}. We briefly recall the construction and some properties, starting with the construction of the quantum polynomial algebra.  We fix $n\geq 1$. Let $R$ be a commutative ring and let $q\in R$. We write $A_{R,q}(n)$ for the $R$-algebra given by generators  $c_{ij}$, $1\leq i,j\leq n$,  and relations:

\begin{align*} c_{ir}c_{is}&=c_{is}c_{ir},  \hbox{ for  } 1\leq i,  r,s\leq n;\cr
c_{jr}c_{is}&=qc_{is}c_{jr}, \hbox{ for } 1\leq i <j \leq n, 1\leq r\leq s < n;\cr
c_{js}c_{ir}&=c_{ir}c_{js}-(q-1)c_{is}c_{jr}, \hbox{ for }  1\leq i<j \leq n, 1\leq r<s\leq n.
\end{align*}

We call the elements  $c_{ij}$ the $(i,j)$ coordinate elements  of $A_{R,q}(n)$.
Since the relations are homogeneous,  $A_{R,q}(n)$ has an $R$-algebra grading 
$A_{R,q}(n)=\bigoplus_{r\geq 0} A_{R,q}(n,r)$ in  which each coordinate element has degree $1$. Then by  \cite{SD25}, Theorem~1.1.8 the elements
\begin{equation*}
	c_{11}^{m_{11}}c_{12}^{m_{12}}\ldots
	c_{1n}^{m_{1n}}c_{21}^{m_{21}}\ldots
	c_{nn}^{m_{nn}},
\end{equation*}
	with  
 $m_{11},\ldots,m_{nn}\geq 0$, form an $R$-basis of $A_{R,q}(n)$. We make this slightly more formal.
 
 \q Let $r\geq 0$. As in \cite{Green},   we write $I(n,r)$ for the set of maps
 $i\colon \{1,\ldots,r\}\to \{1,\ldots,n\}$. We identify $i\in I(n,r)$ with the
 sequence $(i_1,\ldots,i_r)$ in the obvious way. For $i,j\in I(n,r)$ we write $c_{ij}$ for the product
 $c_{i_1j_1}\ldots c_{i_rj_r}$.   We write  $i\leq j$ if $i_a\leq j_a$, for all $1\leq a\leq r$, and write $i<j$ if $i\leq j$ and $i\neq j$.  We write $Y(n,r)$ for the set of all pairs $(i,j)\in I(n,r)$ such that $i_1\leq \cdots\leq i_r$ and whenever, for some  $1\leq a<r$,  we have  $i_a=i_{a+1}$ then $j_a\leq j_{a+1}$. We write $Y(n)$ for the disjoint union of the sets $Y(n,r)$, $r\geq 0$.

\begin{lemma} The elements $c_{ij}$, with $i,j\in Y(n)$ form an $R$-basis of $A_{R,q}(n)$ and, for $r\geq 0$, the elements $c_{ij}$, with $i,j\in Y(n,r)$, form an $R$-basis of $A_{R,q}(n,r)$.

\end{lemma}

\q  We write $I$ for the ideal of $A_{R,q}(n)$ generated by all $c_{ij}$, with $1\leq i<j\leq n$.  We leave it to the reader to check (by an easy induction argument using the defining relations)   the following result.

\begin{lemma} The ideal $I$ has  $R$-basis $c_{ij}$, with $(i,j)\in Y(n,r)$  for some $r$  and $i_a<j_a$, for some $1\leq a\leq r$.

\end{lemma}

\q We set $\barA(n)=A(n)/I$. For $f\in A(n)$ we set $\bar f=f+I\in \barA(n)$.
For $r\geq 0$, we write $\barY(n,r)$ for the set of all $(i,j)\in Y(n,r)$ such that $i\geq j$. We set  $\barY(n)=\bigcup_{r\geq 0} \barY(n,r)$.
As an $R$-module we have $A_{R,q}(n)=I\oplus D$, where $D=\oplus_{(i,j) \in \barY(n)}  Rc_{ij}$. Hence we have the following.

\begin{lemma} $\barA_{R,q}(n)$ has $R$-basis $\barc_{ij}$, $(i,j)\in \barY(n)$,  and, for $r\geq 0$, $\barA_{R,q}(n,r)$ has $R$-basis $\barc_{ij}$, $(i,j)\in \barY(n,r)$.

\end{lemma}

\section{ Quantum general linear groups}

\bs\rm

\q  Let $\K$ be a field. The  category of quantum groups over $\K$  is the dual of the category of Hopf algebras over $\K$.  More informally, we shall use the expression \lq\lq $G$ is a quantum group over $\K$"  to indicate that we have in mind a Hopf algebra over $\K$, which we will denote $\K[G]$ and call the coordinate algebra of $G$. By the expression \lq\lq $\theta\colon G\to H$ is a morphism of quantum groups (over $\K$)"   we indicate  that $G$ and $H$ are quantum groups and that we have in mind a Hopf algebra morphism from $\K[H]$ to $\K[G]$, which we call the comorphism  of $\theta$ and denote $\theta^\sharp$. We shall say that a quantum group $H$ is a (quantum) subgroup of a quantum group $G$ over $\K$ to indicate that $\K[H]=\K[G]/I_H$ for some Hopf ideal $I_H$ of $\K[G]$, which we call the defining ideal of $H$ in $G$. If $H$ is a quantum subgroup of the quantum group $G$ then by the inclusion map  $i\colon H\to G$ we mean the quantum group homomorphism  such that $i^\sharp\colon \K[G]\to \K[H]$ is the natural map.

\q Let $G$ be a quantum group over $\K$. By the category of left $G$-modules we
mean the category of right $\K[G]$-comodules.  We write $\Mod(G)$ for the
category of left $G$-modules and $\mod(G)$ for the category of finite
dimensional left  $G$-modules.  For $V,W\in \Mod(G)$ and $i\geq 0$ we write $\Ext^i_G(V,W)$ for $\Ext^i_{\K[G]}(V,W)$. Let $H$ be a quantum subgroup of $G$. Then we have the induction functor $\ind_H^G=\phi^0\colon \Mod(H)\to \Mod(G)$, where $\phi=i^\sharp$ is the comorphism of the inclusion map $i\colon H\to G$. The functor $\ind_H^G$ is left exact so we have the derived functors $R^i\ind_H^G\colon \Mod(H)\to \Mod(G)$, for $i\geq 0$.

\q We work with the quantum coordinate algebra $A_{R,q}(n)$ of the previous section, now taking $R=\K$ and $q\neq 0$. We write $\Sigma_r$ for the symmetric group on $\{1,2,\ldots,r\}$, for $r$ a positive integer. To simplify notation
we will omit  $\K$ and $q$ in subscript in the objects defined in the previous section, where confusion seems unlikely.

By~\cite{SD25}, Theorem~1.4.2, $A(n)$ has a unique structure of a  bialgebra
with comultiplication $\de\colon A(n)\to A(n)\otimes A(n)$ and counit $\ep\colon A(n)\to \K$, satisfying
$$\de(c_{ij})=\sum_{r=1}^n c_{ir}\otimes c_{rj}, \hskip 20pt \ep(c_{ij})=\de_{ij}$$
for $1\leq i,j\leq n$ and where $\de_{ij}$ is the Kronecker delta. 

\q The quantum determinant 
$$d=\sum_{\pi\in \Sigma_n} \sgn(\pi) c_{1,\pi(1)}c_{1,\pi(2)}\ldots c_{n,\pi(n)} $$
is a group-like element of $A(n)$.  Here $\sgn(\pi)$ denotes  the sign of a
permutation $\pi$. Furthermore, we have $c_{ij}d=q^{i-j} d c_{ij}$ for $1\leq
i,j\leq n$ (see \cite[Section~4]{SD25}). It follows that we can form the Ore localisation  $A(n)_d$.  The
bialgebra structure on $A(n)$ extends to $A(n)_d$ and indeed the localisation
$A(n)_d$ is a Hopf algebra. We write $G(n)$ for the quantum group with coordinate algebra $\K[G(n)]=A(n)_d$.   

\q  We write $B(n)$ for the quantum subgroup whose defining ideal $I_{B(n)}$  is
generated by all $c_{ij}$, with $1\leq i<j\leq n$. We write $T(n)$ for the
quantum subgroup whose defining ideal is generated by all $c_{ij}$ with $1\leq
i,j\leq n$ and $i\neq j$. The inclusion map $A(n)\to \K[G(n)]$ gives rise to an injective map $\barA(n)\to \K[B(n)]$ by which we  identify $\barA(n)$ with a subbialgebra of $\K[B(n)]$.   A $G(n)$-module $V$ is called polynomial (\resp  polynomial of degree $r$) if $\cf(V)\leq A(n)$ (\resp $\cf(V)\leq A(n,r)$) and a $B(n)$-module $M$ is called polynomial (\resp polynomial of degree $r$) if $\cf(M)\leq \barA(n)$ (\resp $\cf(V)\leq \barA(n,r)$). 
We shall often identify a polynomial $G(n)$-module (\resp $B(n)$-module) with the corresponding $A(n)$-comodule (\resp $\barA(n)$-comodule).

\q We shall also need the parabolic (quantum) subgroups containing $B(n)$.  We fix a string $a=(a_1,\ldots,a_m)$  of positive integers whose sum is $n$. We let $I(a)$ be the ideal of $k[G(n)]$ generated by all $c_{ij}$ such that $1\leq i<j\leq a_1$ or $a_1+\cdots + a_r < i< j\leq a_1+\cdots +a_{r+1}$ for some $1\leq r<m$. Then  $I(a)$ is a Hopf ideal and we denote by $P(a)$ the quantum subgroup of $G(n)$ with defining ideal $I(a)$.  Thus we have $P(1,1,\ldots,1)=G(n)$ and $P(n)=B(n)$.  For $1\leq i<m$ we shall write $P_i$ for the \lq \lq minimal parabolic" $P(a)$, where $a=(1,1,\ldots,2,1,\ldots,1)$ (with $2$ in the $i$th position).

\q We now introduce certain combinatorial objects associated with the representation theory of $G(n)$ and its subgroups, following \cite{SD37}. We set $X(n)=\zed^n$. We shall write $\de_n$, or simply $\de$,  for $(n-1,n-2,\ldots,1,0)\in X(n)$. For $1\leq i\leq n$ we set $\ep_i=(0,\ldots,0,1,0,\ldots,0)$ (with $1$ in the $i$th position).  We have the dominance order $\unlhd$ on $X(n)$: for $\lambda=(\lambda_1,\ldots,\lambda_n), \mu=(\mu_1,\ldots,\mu_n)$ we write  $\lambda\unlhd  \mu$ if $\lambda_1+\cdots+\lambda_i\leq \mu_1+\cdots+\mu_i$, for $1\leq i<n$, and $\lambda_1+\cdots+\lambda_n=\mu_1+\cdots+\mu_n$.

\q We  write $X^+(n)$ for the set of all
$\lambda=(\lambda_1,\ldots,\lambda_n)\in X(n)$ with $\lambda_1\geq \cdots\geq
\lambda_n$. Elements of $X(n)$ will sometimes be called weights and elements of
$X^+(n)$ called dominant weights. We write $\Lambda(n)$ for the set of
polynomial weights, i.e., the set of $\lambda=(\lambda_1,\ldots,\lambda_n)\in
X(n)$ with all $\lambda_i\geq 0$, and write $\Lambda^+(n)$ for the set of
dominant weights, i.e., $X^+(n)\bigcap \Lambda(n)$. We define the degree of a
polynomial weight $\lambda=(\lambda_1,\ldots,\lambda_n)$ by
$\deg(\lambda)=\lambda_1+\cdots+\lambda_n$.  For $r\geq 0$ we  define
$\Lambda(n,r)\subset \Lambda(n)$  to be the set of all polynomial weights of
degree $r$ (or compositions of $r$). We define  the length, $\l(\lambda)$, of a polynomial weight $\lambda$ to be  $0$ if $\lambda=0$ and to be the number of non-zero entries of $\lambda$ if $\lambda\neq 0$.

\q For $\lambda=(\lambda_1,\ldots,\lambda_n)\in X(n)$ we have a one dimensional
$B(n)$-module $\K_\lambda$: the comodule  structure map $\tau\colon \K_\lambda\to \K_\lambda\otimes k[B(n)]$ takes $v\in \K_\lambda$ to $v\otimes (c_{11}^{\lambda_1}\ldots c_{nn}^{\lambda_n}+I_{B(n)})$. We regard $\K_\lambda$ also as a $T(n)$-module by restriction. 

The modules $\K_\lambda$, $\lambda\in X(n)$, form a complete set of pairwise
non-isomorphic irreducible $T(n)$-modules. For a $T(n)$-module $V$ we have the
weight space decomposition  $V=\oplus_{\lambda\in X(n)} V^\lambda$, where $V^\lambda$ is a direct sum of copies of $\K_\lambda$, $\lambda\in X(n)$.

\q  For $\lambda\in X(n)$ the induced module $\ind_{B(n)}^{G(n)}\K_\lambda$ is non-zero if and only if $\lambda\in X^+(n)$.  We  set $\nabla(\lambda)=\ind_{B(n)}^{G(n)}\K_\lambda$, for $\lambda\in X^+(n)$.  The socle $L(\lambda)$ of $\nabla(\lambda)$ is simple. The modules $L(\lambda)$, $\lambda\in X^+(n)$, form a complete set of pairwise non-isomorphic irreducible $G(n)$-modules and the modules $L(\lambda)$, $\lambda\in \Lambda^+(n)$, form a complete set of pairwise non-isomorphic irreducible polynomial  $G(n)$-modules. We will write $D$ for the determinant module, i,e., the (one dimensional) left $G(n)$-module $L(1,\ldots,1)$.

\q Let $1\leq i<n$.  Let $\lambda=(\lambda_1,\ldots,\lambda_n)\in X(n)$ and suppose that $m=\lambda_i-\lambda_{i+1}\geq 0$.  We define $\nabla_i(\lambda)=\ind_{B(n)}^{P_i(n)}\K_\lambda$. Then $\nabla_i(\lambda)$ has weights $\lambda-r(\ep_i-\ep_{i+1})$, $0\leq r\leq m$, each occurring with multiplicity one (see \cite{SD37}, p251). 

\q We shall need that a $G(n)$-module whose composition factors  have the form $L(\lambda)$ with $\lambda\in \Lambda^+(n)$ (resp. $\lambda\in \Lambda^+(n,r)$) is polynomial (resp. polynomial of degree $r$). Given the results of \cite{SD37} this follows from the arguments in the classical case in  \cite{SD17}. We make this explicit.

\q  Let $\pi\subseteq X^+(n)$.  We say that a $G(n)$-module $V$ belongs to $\pi$ if each composition factor of $V$ belongs to  $\{L(\lambda)\vert\ \lambda\in \pi\}$.  For an arbitrary $G(n)$-module we write $O_\pi(V)$ 
 for the largest $G(n)$-submodule of $V$ belonging to $\pi$.  Regarding $\K[G(n)]$ as the  left regular $G(n)$-module we define $A(\pi)=O_\pi(\K[G])$.  Then, by the arguments for the classical case, \cite{SD17},  Section 1.2, one has the following.

 \begin{lemma} $A(\pi)$ is a subcoalgebra of $\K[G(n)]$  and a $G(n)$-module $V$ belongs to $\pi$ if and only if $\cf(V)\leq A(\pi)$.
 \end{lemma}

 In the case $\pi=\Lambda^+(n,r)$, $r\geq 0$,  we have $A(\pi)=A(n,r)$, see \cite{SD37}, p263,  and taking $\pi=\Lambda^+(n)$, since $\Lambda^+(n)=\bigcup_{r\geq 0} \Lambda^+(n,r)$, we have $A(\pi)=A(n)$. Hence the above lemma gives:

\begin{lemma}  A $G(n)$-module $V$  is polynomial (resp. polynomial of degree
	$r$)   if and only if each  composition factor  of $V$ belongs to  $\{L(\lambda)\vert \lambda\in \Lambda^+(n)\}$  (resp.  $\{L(\lambda)\vert \lambda\in \Lambda^+(n,r)\}$). 
\end{lemma}

\begin{remark}  We note that if $M$ is a polynomial $B$-module then $\ind_B^GM$ is a polynomial $G$-module. It is enough to check this for $M$ finite dimensional since induction commutes with direct limits. By the left exactness of induction and the  Lemma 4.2  it is enough to check this for $M$ one dimensional. So we may assume that $M=\K_\lambda$ for some $\lambda\in \Lambda(n)$. But now we have 
$$\ind_B^G\K_\lambda=\begin{cases}
\nabla(\lambda), & \hbox{ if } \lambda\in \Lambda^+(n);\cr
0, & \hbox{ if }  \lambda\not\in \Lambda^+(n)
\end{cases}$$
and, in particular,  $\ind_B^GM$ is polynomial.
\end{remark}

\section{Extensions of  $B$-modules  and polynomial \\ $B$-modules}

\q Though it is not needed for the application to resolutions of modules for the
Borel-Schur algebras, we take this opportunity to put on record the quantised
version of \cite[Theorem~7]{WoodcockMasterpiece} giving that, for homogeneous polynomial $B(n)$-modules, the extension groups $\Ext^i(V,X)$ are the same whether calculated in the module category of the Borel-Schur algebra or the full $B(n)$-module category.  Though the proof given here looks rather different it is similar at key points to that of Woodcock in the classical case, \cite{WoodcockBielefeld} and we gratefully acknowledge the influence of  \cite{WoodcockBielefeld}. A later proof was given in \cite{WoodcockComm} using the deep theory of  cohomology of line bundles on Schubert varieties due to van der Kallen, \cite{vdK} and related results are to be found in the later work \cite{WoodcockMasterpiece} using the theory of global bases.

\q In this section we adopt the following notation. We put $B=B(n)$, $A=\barA(n)$ and $A_m=\K[\barc_{m1},\barc_{m2},\ldots,\barc_{mm}]$, $x_m=\barc_{mm}$, $y_m=\barc_{mm}^{-1}$ , for $1\leq m\leq n$.  For $\alpha=(\alpha_1,\ldots,\alpha_n)\in \Lambda(n)$ we put $x^\alpha=x_1^{\alpha_1}\ldots x_n^{\alpha_n}$ and $y^\alpha=y_1^{\alpha_1}\ldots y_n^{\alpha_n}$. We write simply $d$ for the restriction of the determinant to the quantum subgroup $B(n)$,  i.e., $d=x_1\ldots x_n$.

\q We have $A=A_1\otimes \cdots \otimes A_n$ and it is easy to check that 
$$A_m/x_mA_m\cong \begin{cases} A_{m-1}, & \hbox{ for $ 1<m\leq n$};\cr
\K, & \hbox{ for $m=1$}
\end{cases}$$
as $B(n)$-modules. 

\q We shall need the following result.

\begin{lemma} Let $\lambda \in \Lambda(n)$ and suppose $1\leq m\leq n$ is such that $\lambda_m\neq 0$. Let $Z$ be a polynomial $B$-module such that for each weight $\mu$ of $Z$, we have $\mu_m=\mu_{m+1}=\cdots=\mu_n=0$. Then  we have
$$\Ext^i_A(\K_\lambda,Z\otimes A_{m+1}\otimes \cdots \otimes A_n)=0$$
for all $i\geq 0$. In particular, we have 
$$\Ext^i_A(\K_\lambda,A/x_mA)=0$$
for all $i\geq 0$.

\end{lemma}

\begin{proof} Suppose not and let $i$ be minimal for which the lemma fails.  Since $\Ext^i_A(\K_\lambda, -)$ commutes with direct limits, the lemma fails for some finite dimensional $Z$  and by the long exact sequence we may assume that $Z=\K_\mu$, for some $\mu\in \Lambda(n)$, with $\mu_m=\mu_{m+1}=\cdots =\mu_n=0$. Now $\K_\mu\otimes A_{m+1}\otimes \cdots \otimes A_n$ has socle $\K_\mu\otimes \K[x_{m+1},\ldots,x_n]$ and so for each weight $\nu$ of the socle we have $\nu_m=0$. Since 
$\lambda_m\neq 0$ there can be no non-zero image of $\K_\lambda$ in the socle of $\K_\mu\otimes A_{m+1}\otimes \cdots \otimes A_n$ 
and therefore $\Hom_A(\K_\lambda,\K_\mu\otimes A_{m+1}\otimes \cdots \otimes A_n)=0$. Thus we must have $i>0$. 

\q Now we have a short exact sequence of $A$-comodules (or polynomial $B$-modules)
$$0\to \K_\mu \to A_1\otimes\cdots\otimes  A_{m-1}\to Q\to 0$$
and for each weight $\nu$ of  $A_1\otimes\cdots\otimes  A_{m-1}$, and hence $Q$, we have $\nu_m=0$.
Tensoring with $A_{m+1}\otimes \cdots\otimes A_n$ we obtain the short exact sequence 
\begin{align*}0\to \K_\mu\otimes A_{m+1}\otimes \cdots\otimes A_n\to  &A_1\otimes \cdots \otimes A_{m-1}\otimes A_{m+1}\otimes \cdots \otimes A_n\cr
\to & Q\otimes A_{m+1}\otimes \cdots \otimes A_n\ \to 0.
\end{align*}
But now $A_1\otimes \cdots \otimes A_{m-1}\otimes A_{m+1}\otimes \cdots \otimes
A_n$ is an injective $A$-comodule (since it is a direct summand of $A$, viewed as the right regular comodule) so we get
$$\Ext^i_A(\K_\lambda,\K_\mu\otimes A_{m+1}\otimes \cdots\otimes A_n)=\Ext^{i-1}_A(\K_\lambda,Q\otimes A_{m+1}\otimes \cdots\otimes A_n)$$
from the long exact sequence. This is $0$, by the minimality of $i$, and so we are done.
\end{proof}

\q We  now consider the functor $\pol\colon \Mod(B)\to \Comod(A)$, taking $X\in\Mod(B)$ to the largest polynomial submodule of $X$. For a morphism  $\theta\colon X\to X'$, of $B$-modules, $\pol(\theta)\colon \pol(X)\to \pol(X')$ is the restriction of $\theta$.

\q For $V\in \Comod(A)$, $X\in \Mod(B)$, since the image of any $B$-module
homomorphism from $V$ to $X$ is contained in $\pol(X)$, we have
$\Hom_B(V,X)=\Hom_A(V,\pol(X))$.  Thus we get a factorisation of left exact functors
$$\Hom_B(V,-)=\Hom_A(V,-)\circ \pol.$$
Moreover, $\pol(\K[B])=A$ and it follows that $\pol$ takes injective $B$-modules to injective $A$-comodules. Thus, for $V\in \Comod(A)$, $X\in \Mod(B)$, we have a Grothendieck spectral sequence,  with second page $\Ext^i_A(V,R^j\pol X)$, converging to $\Ext^*_B(V,X)$. In particular, if $k>0$ and $R^j \pol (X)=0$ for all $0<j<k$, then we have the $5$-term exact sequence
\begin{equation}
	\label{fiveterm}
	\begin{aligned}0\to \Ext^k_A(V,\pol X)\to &\Ext^k_B(V,X)\to
		\Hom_A(V,R^k\pol X)\cr
		\to &\Ext^{k+1}_A(V,\pol X)\to \Ext^{k+1}_B(V,X).
	\end{aligned}
\end{equation}

\begin{theorem} (i) Let $X$ be a polynomial $B$-module. Then we have \\
$R^i\pol (X)=0$, for all $i>0$.

(ii) If $V$ is also a polynomial $B$-module then the above spectral sequence degenerates and we have $\Ext^i_A(V,X)=\Ext^i_B(V,X)$, for all $i\geq 0$.

\end{theorem}

\begin{proof} For $k>0$ we prove by induction the statement $P(k)$: for all polynomial $B$-modules $V',X'$ we have $R^i\pol X'=0$ for all $0< i < k$  and $\Ext^i_A(V',X')=\Ext^i_B(V',X')$, for all $0\leq i < k$.

\q Note that $P(1)$ is true since $\Hom_A(V',X')=\Hom_B(V',X')$ for polynomial $B$-modules $V',X'$.  We now assume $P(k)$ and deduce $P(k+1)$.   

\q We claim that $R^k\pol A=0$.   Assume, for a contradiction, that this is not the case. Then the $B$-socle of
$R^k\pol A$ is not zero so we have \\
$\Hom_B(\K_\lambda, R^k\pol A)\neq 0$ for some $\lambda\in\Lambda(n)$.  Dimension shifting, using the short exact sequence
$$0\to A\to \K[B]\to \K[B]/A\to 0$$
gives $\Ext^{k-1}_B(\K_\lambda,\K[B]/A)\neq 0$. Now $\K[B]$ has an ascending
exhaustive filtration $A\subseteq d^{-1} A\subseteq d^{-2} A\subseteq \cdots$
and $\Ext^{k-1}_B(\K_\lambda, -)$ commutes with direct limits so we must have
$\Ext^{k-1}_B(\K_\lambda, d^{-s} A/A)\neq 0$ for some $s>0$.  Hence we have
$\Ext_B^{k-1}(\K_\lambda,y^\alpha A/A)\neq 0$, for some $\alpha\in \Lambda(n)$.
We choose $\alpha,\beta\in \Lambda(n)$ with $\beta_i\leq \alpha_i$, for $1\leq
i\leq n$, and with $\deg(\alpha)-\deg(\beta)$ minimal subject to the condition
$\Ext^{k-1}_B(\K_\lambda,y^\alpha A/y^\beta A)\neq 0$. Note that in fact we must
have $\deg(\alpha)=\deg(\beta)+1$ since if $\gamma\in \Lambda(n)$ with
$\beta_i\leq \gamma_i\leq \alpha_i$, for all $i$, then we get a short exact sequence
$$0\to y^\gamma A/y^\beta A\to y^\alpha A/y^\beta A\to y^\alpha A/y^\gamma A\to 0$$
and so we must have 
\begin{equation*}
	\Ext^{k-1}_B(\K_\lambda,y^\gamma A/y^\beta A)\neq 0 \mbox{ or }
	\Ext^{k-1}_B(\K_\lambda,y^\alpha A/y^\gamma A)\neq 0.
\end{equation*}
	Thus we have $\alpha=\beta+\ep_m$, for some $1\leq m\leq n$. Hence we have $\Ext^{k-1}_B(\K_\lambda, y^{\beta+\ep_m}A/y^\beta A)\neq 0$ and so 
$$\Ext^{k-1}_B(\K_\lambda\otimes \K_{\beta+\ep_m}, \K_{\beta+\ep_m} \otimes (y^{\beta+\ep_m} A/y^\beta A))\neq 0.$$

Thus we have $\Ext^{k-1}_B(\K_{\nu},A/x_mA)\neq 0$, where $\nu=\lambda+\beta+\ep_m$. By the inductive hypothesis, we have  
\begin{equation*}
	\Ext^{k-1}_A(\K_\nu, A/x_mA)=\Ext^{k-1}_B(\K_{\nu},A/x_mA) \neq 0
\end{equation*}
and this contradicts Lemma~5.1.  Hence we have $R^k\pol A=0$.  Since $R^k\pol$ commutes with direct limits we also have $R^k\pol Z=0$, where $Z$ is a direct sum of copies of the right regular comodule $A$. Let $X'$ be any  polynomial $B$-module. Then $X'$ embeds in a direct sum of copies of $A$, via the comodule structure map. Thus we have a short exact sequence $0\to X'\to Z\to Y\to 0$, where $Z$ is a direct sum of copies of  $A$ and $Y$ is a polynomial $B$-module.  Now the derived functors of $\pol$ give the exact sequence 
$$R^{k-1}\pol Y \to R^k \pol X'\to R^k\pol Z=0.$$
But also, we have $R^{k-1}\pol Y=0$, from the inductive hypothesis, so that $R^k\pol X'=0$. 

\q Now for $V',X'\in \Comod(A)$ the $5$-term exact sequence~\eqref{fiveterm}  gives an isomorphism $\Ext^k_A(V',X')\to \Ext^k_B(V',X')$.
This completes the proof of $P(k+1)$. Hence $P(k)$ is true for all $k$. Thus we have $R^i\pol X=0$ for all $i>0$.  This proves (i). 

 (ii) follows from (i).

\end{proof}

\begin{corollary} A $B$-module is polynomial if and only if all its weights are polynomial.
\end{corollary}

\begin{proof} If $V$ is a polynomial $B$-module then $V$ embeds, via the comodule structure map, into a direct sum of copies of $A(n)$ and it follows that all weights of $V$ are polynomial. To prove that a $B$-module $V$ with all weights polynomial is polynomial it suffices, by local finiteness, to consider the case in which $V$ is finite dimensional. If $V$ is one dimensional then it is  isomorphic to $\K_\lambda$, for some $\lambda\in \Lambda(n)$, and hence polynomial. Suppose now that $V$ has dimension bigger than one and let $L$ be a one dimensional submodule. We may assume inductively  that $V/L$ is polynomial. We have a natural isomorphism $\Ext^1_{A(n)}(V/L,L)\to \Ext^1_B(V/L,L)$, by the theorem and it follows that every extension of $V/L$ by $L$ arises from an  $A(n)$-comodule, in particular $V$ is polynomial.
\end{proof}

Let $r\geq 0$. We define the negative (quantised) Borel-Schur algebra  $S^-(n,r)$ to be the  dual algebra of $\barA(n,r)$. We now obtain the quantised version of a theorem of Woodcock, \cite{WoodcockBielefeld}, Theorem 7.

\begin{corollary} Let $V$ and $X$ be polynomial $B$-modules which are homogeneous of degree $r$.  Then we have $\Ext^i_{S^-(n,r)}(V,X)=\Ext^i_B(V,X)$, for all $i\geq 0$.

\end{corollary}

\begin{proof} We have 
$$\Ext^i_{S^-(n,r)}(V,X)=\Ext^i_{\barA(n,r)}(V,X)=\Ext^i_{\barA(n)}(V,X)=\Ext^i_B(V,X).$$

\end{proof}

\section{A vanishing theorem for polynomial \\ {modules}}

\q To save on notation we shall abbreviate    $G(n),B(n),T(n)$ to $G,B,T$ where confusion seems unlikely.

\q We shall need a  bound for the vanishing of $R^i\ind_B^G \K_\lambda$, for $\lambda\in \Lambda(n)$.  We do this by an inductive argument using the function $b$ that we now introduce. For each $\lambda\in \Lambda(n)$ we shall define a non-negative integer $b(\lambda)$.  We define $b$ on $\Lambda(n,r)$, for $r\geq 0$ by descending induction on the dominance order.  If $\lambda$ is dominant or if  $\lambda_j-\lambda_{j+1}=-1$ for some $1\leq j<n$ we set  $b(\lambda)=0$. In particular this defines $b(\lambda)$ for $\lambda=(r,0,\ldots,0)$.  If $\lambda\in \Lambda(n,r)$ is not of the form already considered then we have $\lambda_j-\lambda_{j+1}=-m_j$, with $m_j\geq 2$, for some $1\leq j<n$. We define 
$$b_j(\lambda)=\max\{ b(\lambda+t(\ep_j-\ep_{j+1}))\vert 0< t  < m_j\}+1.$$
and 
\begin{equation}
	\label{blambda}
	b(\lambda)=\min\{b_j(\lambda) \vert 1\leq j<n,
	\lambda_j-\lambda_{j+1}\leq -2\}.
\end{equation}
	By an easy induction one sees that if $\lambda=(\lambda_1,\ldots,\lambda_m,0,\ldots,0)$ with \\
 $\lambda_1,\ldots,\lambda_m\neq 0$ then $b(\lambda)=b(\mu)$, where $\mu=(\lambda_1-1,\ldots,\lambda_m-1,0,\ldots,0)$. 

\begin{lemma} For $\lambda\in \Lambda(n)$ we have $b(\lambda)\leq \deg(\lambda)-\l(\lambda)$.

\end{lemma}

\begin{proof}    Since  $\Lambda(1)$ consists of dominant weights the result holds for $n=1$. Suppose that it is false and let $n$ be minimal for which it fails.
Let  $\lambda\in \Lambda(n)$ be a counterexample of smallest possible degree . If $\lambda=(\lambda_1,\ldots,\lambda_m,0,\ldots,0)$ with $\lambda_1,\ldots,\lambda_m\neq 0$ then
\begin{align*}
	b(\lambda)&=b(\lambda_1-1,\ldots,\lambda_m-1,0,\ldots,0)\cr
&\leq \deg(\lambda_1-1,\ldots,\lambda_m-1,0,\ldots,0)
\cr &\phantom{\le\le} - \l(\lambda_1-1,\ldots,\lambda_m-1,0,\ldots,0)\cr
&= \deg(\lambda)-m-\l(\lambda_1-1,\ldots,\lambda_m-1,0,\ldots,0)\cr
&\leq \deg(\lambda)-m=\deg(\lambda)-\l(\lambda).
\end{align*}
Hence there exists some $1\leq j<n$ such that $\lambda_j=0$, $\lambda_{j+1}> 0$.  If $\lambda_{j+1}=1$ then $b(\lambda)=0$ and $\lambda$ is not a counterexample. Hence we have $\lambda_{j+1}=m\geq 2$.  But now we have
\begin{align*}b(\lambda)&\leq b_j(\lambda)\cr
&=\max\{b(\lambda+t(\ep_i-\ep_{i+1})\vert 0< t < m\}+1.
\end{align*}
We consider $\mu=\lambda+t(\ep_j-\ep_{j+1})$ with $0<t<m$.   Note that $\mu$ has entry $t\neq 0$ in the $j$th position and entry $\lambda_{j+1}-t\geq \lambda_{j+1}-(m-1)=1$ in the $(j+1)$st position. Moreover, $\lambda$ and $\mu$ agree in all positions other than $j$ and $j+1$. Hence we have $\l(\mu)=\l(\lambda)+1$. Moreover, $\mu$ is greater than $\lambda$, in the dominance order. Hence we have 
$$b(\mu)\leq \deg(\mu)- \l(\mu)=\deg(\lambda)-\l(\lambda)-1$$
i.e., $b(\mu)+1\leq \deg(\lambda)-\l(\lambda)$.

\q Since this is true for all $\mu$ of the form $\lambda+t(\ep_i-\ep_{i+1})$
with $0<t<m$,  from~\eqref{blambda},  we have $b(\lambda)\leq \deg(\lambda)-\l(\lambda)$.

\end{proof}

\begin{lemma} (i) For $\lambda\in \Lambda(n)$ we have $R^i\ind_B^G \K_\lambda=0$ for all $i>b(\lambda)$, and hence for $i\geq \deg(\lambda)>0$.

(ii) If $V$ is a polynomial  $B$-module of degree $r$ then $R^i\ind_B^G V=0$ for $i>r$.
\end{lemma}

\begin{proof} (i) We argue by induction on $b(\lambda)$. If $b(\lambda)=0$ then either $\lambda$ is dominant or $\lambda_j-\lambda_{j+1}=-1$ for some $1\leq j<n$, and $R^i\ind_B^G\K_\lambda=0$ for $i>0$, \cite{SD37}, Theorem 3.4 and Lemma 3.1, (ii), 
 and the result holds. So suppose $b(\lambda)>0$ and the result holds for all $\mu\in \Lambda(n)$ with $b(\mu)<b(\lambda)$. We have $b(\lambda)=b_j(\lambda)+1$ for some $1\leq j<n$ with $\lambda_j-\lambda_{j+1}=-m$, $m\geq 2$. Consider the module $\nabla_j(\lambda+(m-1)(\ep_j-\ep_{j+1})+\delta)$.
Writing $\mu=\lambda+(m-1)(\ep_j-\ep_{j+1})+\delta$ we have 
$$\mu_j-\mu_{j+1}=-m+2(m-1)+1=m-1.$$
  Hence $\nabla_j(\lambda+(m-1)(\ep_j-\ep_{j+1})+\delta)$ has weights $\lambda+(m-1)(\ep_j-\ep_{j+1})+\delta$, $\lambda+(m-2)(\ep_j-\ep_{j+1})+\delta$, $\ldots$, $\lambda+\de$, each occurring with multiplicity $1$. Hence the module $\nabla_j(\lambda+(m-1)(\ep_j-\ep_{j+1}+\de)\otimes \K_{-\de}$ has bottom weight $\lambda$ and we have a short exact sequence of $B$-modules
  \begin{equation}
	  \label{klambda}
   0\to \K_\lambda\to \nabla_j(\lambda+m(\ep_j-\ep_{j+1})+\delta)\otimes
   \K_{-\delta}\to Q\to 0.
   \end{equation}
	  where $Q$ has weights $\lambda+i(\ep_j-\ep_{j+1})$, with $1\leq i\leq m-1$.

\q Now  we have $R^i\ind_B^{P_j} \K_{-\delta}=0$ for all $i$, by \cite{SD37}, Lemma~3.1~(ii) and so by the 
 tensor identity and \cite{SD37}, Proposition~1.3~(iii), we have 
 $$R^i\ind_B^{P_j}(\nabla_j(\lambda+(m-1)(\ep_j-\ep_{j+1})+\delta)\otimes \K_{-\delta})=0$$
 for all $i$.   By the spectral sequence arising from the  transitivity of
 induction, \cite{SD37}, Proposition 1.2, we get
 $R^i\ind_B^G(\nabla_j(\lambda+m(\ep_j-\ep_{j+1})+\delta)\otimes
 \K_{-\delta})=0$ for all $i$.   Hence from~\eqref{klambda}  we get $R^i\ind_B^G \K_\lambda= R^{i-1}\ind_B^G Q$.
 
 \q But a   weight $\nu$ of $Q$ has the form $\lambda+t(\ep_j-\ep_{j+1})$, with $1\leq t\leq m-1$, and $b(\nu)\leq b(\lambda)-1$. So that for $i>b(\lambda)$ we have $i-1>b(\nu)$ and hence $R^{i-1}\ind_B^G \K_\nu=0$, by the inductive hypothesis. Since this holds for all weights of $Q$, i.e., for all composition factors $\K_\nu$ of $Q$,  we get $R^{i-1}\ind_B^GQ=0$, from the long exact sequence,  and hence $R^i\ind_B^G \K_\lambda=0$

(ii) This follows from (i) and the long exact sequence.

\end{proof}

\section{Kempf vanishing for quantised  Schur algebras}

\q At this point we introduce the natural left $G$-module for use  later in
this section. 
We write $E$ for the $\K$-vector space with basis $e_1,\ldots,e_n$. Then $E$ is
a $G$-module via the comodule structure map $\tau\colon E\to E\otimes \K[G]$ defined by $\tau(e_i)=\sum_{j=1}^n e_j\otimes c_{ji}$, $1\leq i\leq n$. We shall also need the symmetric powers $S^rE$ and exterior powers $\tbw^rE$ of $E$. We recall the construction from \cite{SD25} and \cite{SD42}.  Let $T(E)$ be the tensor algebra $\oplus_{r\geq 0} E^{\otimes r}$. Thus $T(E)$ is a  graded $\K$-algebra, in such a way that each $e_i\in E$ has degree $1$.
The ideal generated by all $e_ie_j-e_je_i$, $1\leq i,j\leq n$, is homogeneous
and is a $G$-submodule, so the (usual) symmetric algebra $S(E)$ inherits a grading 
$S(E)=\bigoplus_{r\geq 0} S^r(E)$ and each $S^r(E)$ is a $G$-submodule of $S(E)$.
 Also, the ideal of $T(E)$ generated by the elements $e_i^2, e_ke_l+qe_le_k$, $1\leq i\leq n$, $1\leq k<l\leq n$, is homogeneous and a $G$-submodule and we write $\tbw(E)$ for the quotient algebra. Thus $\tbw(E)$ inherits  a grading $\tbw(E)=\oplus_{r\geq 0} \tbw^r(E)$ and each $\tbw^r(E)$ is a $G$-submodule.
 
 \q For $i=(i_1,\ldots,i_r)\in I(n,r)$ we write  $e_i$ for  $e_{i_1}\otimes
 \cdots \otimes e_{i_r}\in E^{\otimes r}$ and $\hate_i$ for the  image of $e_i$  in
 $\tbw^r(E)$. The module $\tbw^r(E)$ has basis $\hate_i$, with $i\in I(n,r)$, running over all maps with $i_1>\cdots > i_r$.

\begin{proposition} (i) Let $1\leq r\leq n$ and let $L_r$ be the simple $B$-module with weight $\ep_r$. Then for any $B$-module $M$ the $\K$-linear map 
$\phi_M\colon M\otimes L_r\to L_r\otimes M$ given by 
$$\phi(m\otimes l)=q^{\alpha_1+\cdots +\ \alpha_r} l\otimes m$$
for $\alpha=(\alpha_1,\ldots,\alpha_n) \in X(n)$ and $m\in M^\alpha$, 
is a $B$-module isomorphism.

(ii) For any $B$-module $M$ and one dimensional $B$-module $L$  the $B$-modules $M\otimes L$ and $L\otimes M$ are isomorphic.

\end{proposition}

\begin{proof} (i) Certainly $\phi_M$ is a linear isomorphism so it remains to show that it is $B$-module homomorphism.  We shall call a $B$-module $M$ admissible if $\phi_M$ is a $B$-module homomorphism. So the point is to show that all $B$-modules are admissible.  
Note also that admissibility is preserved by isomorphism.  Let $M$ and $N$ be $B$-modules. Then the map $\phi_{M\otimes N}\colon  M\otimes N \otimes L \to L\otimes M\otimes N$ factorizes as $(\phi_M\otimes \id_N)\circ (\id_M\otimes \phi_N)$ so that admissibility is preserved under tensor products.  

\q Suppose now that $M$ is a submodule of $N$.  Then  the map $\phi_N\colon N\otimes L_r\to L_r\otimes N$ restricts to $\phi_M\colon M\otimes L_r\to M\otimes L_r$. Thus if $N$ is admissible then so is $M$.  Similarly, if $N$ is admissible then so is the quotient $N/M$.   By the  local finiteness of $B$-modules it suffices to prove that finite dimensional $B$-modules are admissible.  

\q We now prove that all $G$-modules are admissible.   Note that if $M$ is one dimensional then the twisting map $\theta\colon M\otimes  L_r\to L_r\otimes  M$, given by 
$\theta(m\otimes l)=(l\otimes m)$, for $l\in L_r$, $m\in M$, is a $B$-module map and $\phi_M$ is a scalar multiple of $\theta$. Hence one dimensional $B$-modules are admissible. 
Now if $M$ is any finite dimensional $G$-module then  $D^{\otimes r}\otimes M$ is polynomial for some $r\geq 0$. Hence $M$ is isomorphic to a module of the form 
$Z\otimes N$, where $Z$ is the dual of $D^{\otimes r}$ and $N$ is polynomial. Hence it suffices to prove that finite dimensional polynomial modules are admissible. 

\q Note that if $M=M_1\oplus \cdots \oplus M_t$, for $G$-modules $M_1,\ldots,M_t$ then $\phi_M=\phi_{M_1}\oplus\cdots \oplus \phi_{M_t}$ so that if each $M_i$ is admissible then so is $M$. Now if $M$ is a polynomial $G$-module then, for some $r\geq 0$,  we may write $M=M_0 \oplus\cdots \oplus M_s$, where $M_r$ is polynomial of degree $r$, for $0\leq r\leq s$. Hence it suffices to prove that for each $r$, all polynomial $G$-modules of degree $r$ are admissible. 

\q We now check that the natural module $E$ is admissible.  Let $Z$ be the subspace of $\K[B]$ spanned by the elements $\barc_{ni} \barc_{rr}$, $1\leq i\leq n$. It is  seen from the defining relations that  $Z$ is  also spanned by $\barc_{rr}\barc_{ni}$, $1\leq i\leq n$.  We fix a non-zero element $l_0$ of $L_r$. 
The subspace $Z$ is a left $B$-submodule of $\K[B]$ and we have $B$-module
isomorphisms $\theta\colon  E\otimes L_r\to  Z$, $\eta\colon L_r \otimes E\to Z$ satisfying  
$\theta(e_i\otimes l_0)=\barc_{ni}\barc_{rr}$ and 
$\eta(l_0\otimes e_i)=\barc_{rr}\barc_{ni}$, for $1\leq i\leq n$.

\q Hence we have an isomorphism $\psi=\eta^{-1}\circ \theta\colon E\otimes L_r \to L_r\otimes E$.  We consider first the case $r=n$. The element  $\barc_{nn}$ commutes with the elements $\barc_{ni}$ and $\psi\colon E\otimes L_n \to  L_n\otimes E$ is the twisting map, taking $e_i\otimes l_0\to l_0\otimes e_i$, for $1\leq i\leq n$. The map $\phi_E$ is $q\psi$ and hence is a homomorphism. Now suppose that $r<n$.   For $1\leq i\leq r$ we have $\theta(e_i\otimes l_0)=\barc_{ni}\barc_{rr}=q\barc_{rr}\barc_{ni}=\eta(q l_0\otimes e_i)$ and hence $\psi( e_i\otimes l_0)=q l_0\otimes e_i$. For $i>r$ we have $\theta(e_i\otimes l_0)=\barc_{rr}\barc_{ni}=\barc_{ni}\barc_{rr}$ (from the defining relations) and so 
$\theta(e_i\otimes l_0)=\eta(l_0\otimes e_i)$.  We have shown that the $B$-module homomorphism $\psi\colon E\otimes L_r \to L_r\otimes E$ is given by
$$\psi(e_i\otimes l_0)=\begin{cases} q l_0\otimes e_i, & \hbox{ if } 1\leq i\leq r;\cr
l_0\otimes e_i, & \hbox{ if } r< i\leq n.
\end{cases}$$
Thus $\psi=\phi_E$ and therefore $\phi_E$ is a $B$-module homomorphism.

\q Since the class of admissible modules is closed under taking tensor products and quotients  we get that the $j$th symmetric power $S^jE$ is admissible for all $j\geq 0$. Now we get that for any $\beta=(\beta_1,\ldots,\beta_n)\in \Lambda(n)$ the module $S^\beta E=S^{\beta_1}E\otimes \cdots\otimes  S^{\beta_n}E$ is admissible. But these modules $S^\beta E$ are injective in the category of polynomial $G$-modules and every finite dimensional polynomial $G$-module embeds in a direct sum of copies of the modules $S^\beta E$, \cite{SD42}, Section 2.1. Hence every finite dimensional polynomial module is admissible and hence all $G$-modules are admissible. 

\q The left regular $B$-module $\K[B]$ is admissible since it is the image of  the restriction homomorphism $\K[G]\to \K[B]$. Hence a direct sum of copies of $\K[B]$ is admissible.  Let $M$ be a $B$-module. Then  the structure map $\tau\colon M\to M\otimes \K[B]$ embeds $M$ into a direct sum of copies of the left regular $B$-module and hence $M$ is admissible. This complete the proof of (i).

(ii) We have $L=\K_\lambda$, for some $\lambda\in\Lambda(n)$.    If $\lambda=0$ there is nothing to prove.  If $\lambda\neq 0$ and $\lambda\in \Lambda(n)$  we may write $\lambda=\mu+\ep_r$, for some $1\leq r\leq n$ and $\mu\in \Lambda(n)$.  Then we have $\K_\lambda  \cong  \K_\mu\otimes L_r$ so  we get 
$$\K_\lambda \otimes M\cong \K_\mu\otimes L_r\otimes M\cong \K_\mu\otimes M\otimes L_r$$
 by part (i) and now it follows by induction on degree that $\K_\lambda \otimes M$ is isomorphic to $M\otimes \K_\lambda$.  Finally, if $\lambda=\mu-\tau$, where $\mu,\tau\in \Lambda(n)$,  then we have that 
 \begin{align*}\K_\tau \otimes \K_\lambda \otimes M&\cong \K_\mu\otimes M\cong M\otimes \K_\mu\cr
 &\cong M\otimes \K_\tau\otimes \K_\lambda \cong \K_\tau\otimes M\otimes \K_\lambda.
 \end{align*}
  So we have that $\K_\tau\otimes \K_\lambda \otimes M$ is isomorphic to $\K_\tau\otimes M\otimes \K_\lambda$ and tensoring on the left with the dual of $\K_\tau$ gives the desired  result.

\end{proof}

\begin{corollary}   Let $t\geq 1$. Let $V_j$ be a polynomial $B$-module of degree $r_j$, for $1\leq j\leq t$ and let $M_j$ be a $G$-module for $1\leq j\leq t+1$. Then we have
$$R^i\ind_B^G(M_1\otimes V_1\otimes \cdots\otimes M_t\otimes V_t\otimes M_{t+1})=0$$
for $i>r_1+\cdots+r_t$.

\end{corollary}

\begin{proof} By the long exact sequence we may assume that each $V_i$ is one
	dimensional. Then by the Proposition $M_1\otimes V_1\otimes
	\cdots\otimes M_t\otimes V_t\otimes M_{t+1}$ is isomorphic to $M\otimes
	V$, where $M=M_1\otimes \cdots \otimes M_{t+1}$ and $V=V_1\otimes
	\cdots\otimes  V_t$ and so the result follows from the tensor identity and 
Lemma~6.2(ii).

\end{proof}

\begin{remark} Recall (or check, by dimension shifting)  that if $m\geq 0$ and 
$$0\to X_r\to X_{r-1}\to \cdots \to X_0\to M\to 0$$
 is an exact sequence of $B$-modules such that $R^i\ind_B^GX_j=0$ for all $i>m+j$ then $R^i\ind_B^G M=0$ for all $i>m$.  In particular if $R^i\ind_B^G X_j=0$ for all $i>j$ then $R^i\ind_B^G M=0$ for all $i>0$.
 \end{remark}

\begin{proposition}
\label{indovAisA}
We have
$$R^i\ind_B^G \barA(n)=\begin{cases} A(n), & \hbox{\rm if } i=0;\cr
0, & \hbox{\rm if } i>0.
\end{cases}$$

\end{proposition}

\begin{proof}  We first consider the case $i=0$. The natural map $\pi\colon  A(n)\to \barA(n)$  gives rise to a $G$-module map $\tilde\pi\colon A(n)\to \ind_B^G \barA(n)$, given by 
$$\tilde\pi(f)=\sum_{i=1}^m \pi(f_i)\otimes f_i'$$
where $\de_{\K[G]}(f)=\sum_{i=1}^m f_i\otimes f_i'\in  A(n)\otimes \K[G]$. Now
if $\tilde\pi(f)=0$ then applying $\ep_{\K [B]}\otimes \id_{\K [G]}$ we get 
$$0=\sum_{i=1}^m \ep_{\K [B]}\pi(f_i)f_i'=\sum_{i=1}^m \ep_{\K [G]}(f_i)f_i'=f.$$
Hence $\tilde\pi$ is injective. Now the inclusion $\barA(n)$ of $\K [B]$ gives
rise to an injective $G$-module homomorphism $\ind_B^G \barA(n)\to \ind_B^G
\K[B]=\K [G]$. Moreover, $\ind_B^G \barA(n)$ is polynomial, by Remark 4.3,  so that this map goes onto $A(n)$. Hence we have a composition of injective $G$-module homomorphisms
$$A(n)\to \ind_B^G \barA(n)\to A(n).$$
But now, restricting to degree  $r$ we get an injective homomorphism $A(n,r)\to A(n,r)$ and since $A(n,r)$ is finite dimensional this map is surjective, for all $r\geq 0$. Hence the composite $A(n)\to \ind_B^G \barA(n)\to A(n)$ is surjective and the second map $\ind_B^G\barA(n)\to A(n)$ is a $G$-module isomorphism.

\q  We now suppose $i>0$.  Let $A_r$ be the subaglebra of $\barA(n)$ generated by the elements $\barc_{r1},\ldots,\barc_{rr}$. Then $A_r$ is a $B$-submodule and  the multiplication map $A_1\otimes \cdots \otimes A_n\to \barA(n)$ is a $B$-module isomorphism. 

\q Let  $0\leq m<  n$ and  let $V_m$ be the $B$-submodule of $E$ spanned by $e_j$, with $m<j\leq n$.  Let  $V=V_m$.  For $r\geq 0$ we write $\tbw^r V$ for the subspace of $\tbw^r E$ spanned by $\hate_i$, with $i\in I(n,r)$, $i_a>m$ for $1\leq a\leq r$. Since this is the image of the $B$-submodule $ V^{\otimes r}$ of $E^{\otimes r}$ under the natural map $E^{\otimes r}\to \tbw^r E$ we have that $\tbw^r V$ is a $B$-submodule of $E^{\otimes r}$. Similarly we have that the ideal $J$, say,  of $S(E)$ generated by $V$ is a $B$-submodule. We write $S(E/V)$ for the $\K$-algebra and $B$-module $S(E)/J$.  We write $S^r(E/V)$ for the $r$th homogeneous component of $S(E/V)$.
 
 Let $U_m$ be the $B$-submodule of $A(n)$ spanned by $\barc_{m1},\ldots, \barc_{ mm}$. Now we have a $B$-module homomorphism $\theta\colon  E\to U_m$, sending $e_i$ to $\barc_{mi}$, for $1\leq i\leq m$, and to $0$ for $m<i\leq n$.  Then $\theta$ induces a $B$-module isomorphism 
$S(E/V_m)\to A_m$.  Hence we have
\begin{equation}
	\label{barAnr}
	\barA(n,r)\cong \bigoplus_{r=r_1+\cdots+r_n}
	S^{r_1}(E/V_1)\otimes\cdots\otimes  S^{r_n}(E/V_n).
\end{equation}

 By  \cite{SD37}, Lemma 3.3(ii),  we have that, for $r>0$, the $\K$-linear map $\psi\colon \tbw^rE\to E\otimes \tbw^{r-1}E$, given by 
$$\psi(\hate_i)=\sum_{a=1}^r (-1)^{a-1} e_{i_a}\otimes \hate_{i_1\ldots {\hat i_a}\ldots i_r}$$
for $\in I(n,r)$ with $i_1>\cdots>i_r$ 
(where ${\hat i_a}$ indicates that $i_a$ is omitted) is a $G$-module homomorphism. Combining these maps with the multiplication maps $E\otimes S^b(E)\to S^{b+1}(E)$, $b\geq 0$, (also $G$-module maps) in the usual way, for $a\geq 0$,  we obtain  the Koszul resolution
$$0\to \tbw^aE\to \cdots \to \tbw^2 E \otimes  S^{a-2}E\to E\otimes  S^{a-1}E \to S^r(E)\to 0.$$
By restricting the maps in the above we obtain, in the usual way, 
the Koszul resolution 
(cf.\cite[II~12.12 (i)]{RAG})
\begin{align*}0\to \tbw^a V_m&\to \cdots\to  \tbw^{a-j}V_m \otimes S^j V_m\cr
&\to \cdots \to S^a(E) \to S^a(E/V_m)\to 0.
\end{align*}
Tensoring all such together, for $1\leq m\leq n$,  we obtain a resolution 
$$\cdots \to Y_1\to Y_0\to S^{r_1}(E/V_1)\otimes\cdots\otimes  S^{r_n}(E/V_n)\to 0$$
where each term $Y_s$ is a direct sum of modules of the form   
$M_1\otimes Z_1\otimes \cdots \otimes M_t\otimes Z_t$ with each $M_i$ a $G$-module and each $Z_j$ polynomial of degree $d_j$, say, $d_1+\cdots+d_t=s$.  Now from Corollary~7.2, 
we have that $R^i\ind_B^G Y_s=0$ for $i>s$ and hence by Remark~7.3 we have
$R^i(S^{r_1}(E/V_1)\otimes\cdots\otimes  S^{r_n}(E/V_n))=0$ for all $i>0$. Hence
by~\eqref{barAnr}  above we have $R^i\ind_B^G \barA(n)=0$, for all $i>0$.

\end{proof}

\bs

\begin{theorem}
\label{commutativity}
Let $\phi\colon \barA(n)\to k[B]$ and $\psi\colon A(n)\to k[G]$ be the
inclusion maps.  Let $\pi\colon A(n)\to \barA(n)$ be the restriction map.  Then, for  $V\in \Comod(\barA(n))$,  we have  $R^i\ind_B^G( \phi_0 V)\cong \psi_0 R^i\pi^0 V$ for all $i\geq 0$.
\end{theorem}

\begin{proof}   If  $M$ is  a polynomial $B$-module then $\ind_B^GM$ is a
polynomial $G$-module, by Remark 4.3. Hence we have $\ind_B^G\circ
\phi_0= \psi_0\circ \pi^0\colon \Comod(\barA(n))\to \Mod(G)$. We write
$F=\ind_B^G\circ \phi_0= \psi_0\circ \pi^0$.   Now $\psi_0$ is exact so
we have $RF^i V =\psi_0 R^i\pi^0 V$.  An  injective $\barA(n)$-comodule
is a direct summand of a direct sum of copies of the left regular
comodule $\barA(n)$. So it follows from the Proposition~\ref{indovAisA}
that $\phi_0$ takes injective objects to $\ind_B^G$-acyclic objects.
Hence we have a Grothendieck spectral sequence, with second page
$R^i\ind_B^G\circ R^j\phi_0 V$ converging to $R^*F V$. But $\phi_0$ is
exact, so the spectral sequence degenerates and we have $R^i F V\cong R^i\ind_B^G( \phi_0 V)=\psi_0 R^i\pi^0 V$.  

\end{proof}

\begin{remark} Slightly less formally, identifying  $\Comod(\barA(n))$ with the
full subcategory  of $B$-modules whose objects are the  polynomial
modules and identifying  $\Comod(A(n))$ with the subcategory of
$G$-modules whose objects are the polynomial modules, we have $R^i \pi^0 V \cong R^i\ind_B^G V$ for a polynomial $B$-module $V$.

\end{remark}

\q The Theorem~\ref{commutativity} has the following corollary, generalising Remark 4.3,   but  which may also be proved by a straightforward dimensional shifting argument.

\begin{corollary}  If $V$ is a polynomial $B$-module then $R^i\ind_B^G V$ is a polynomial $G$-module, for all $i\geq 0$.

\end{corollary}

\q However, the main point of the discussion is to demonstrate the following result, which follows from Kempf's Vanishing Theorem for $G$, as in \cite{SD37},  Theorem 3.4.

\begin{corollary} (Kempf Vanishing  for polynomial modules.)  Let $\pi\colon A(n)\to \barA(n)$ be the restriction map.  For  $\lambda\in \Lambda^+(n)$ we have 
$$R^i\pi^0 \K_\lambda=
\begin{cases}
\nabla(\lambda), & \hbox{ if } i=0;\cr
0, & \hbox { if } i>0.
\end{cases}$$
\end{corollary}

\q Let $\pi(r)\colon A(n,r)\to \barA(n,r)$ be the restriction of $\pi$.  Now
$\pi=\oplus_{r=0}^\infty  \pi(r)\colon A(n)=\oplus_{r=0}^\infty A(n,r)\to \barA(n,r)$. If $V\in \Comod(\barA(n))$ then we may write $V$ uniquely as $V=\bigoplus_{r=0}^\infty V(r)$, where $V(r)\in \Comod(\barA(n,r))$ (or less formally, $V(r)$ is polynomial of degree $r$). It follows that 
$R^i\pi^0 V=\bigoplus_{r=0}^{\infty} R^i\pi(r)^0 V(r)$. Hence we get:

\begin{corollary} (Kempf  Vanishing  for homogeneous polynomial modules.)  Let
	$r\geq 0$ and let  $\pi(n,r)\colon A(n,r)\to \barA(n,r)$ be the restriction map.  For  $\lambda\in \Lambda^+(n,r)$ we have 
$$R^i\pi(n,r)^0 \K_\lambda=
\begin{cases}
\nabla(\lambda), & \hbox{ if } i=0;\cr
0, & \hbox { if } i>0.
\end{cases}$$
\end{corollary}

Let $S(n,r)=A(n,r)^*$ and $S^-(n,r)=\barA(n,r)^*$. Then  from Proposition~2.1   we get :

\begin{corollary} (Kempf Vanishing for Schur algebras)  For $\lambda\in \Lambda^+(n,r)$ we have
$$\Tor_i^{S^-(n,r)}(\K_\lambda^*,S(n,r))=
\begin{cases} \nabla(\lambda)^*, & \hbox{ if } i=0;\cr
0,  &\hbox{ if } i>0.

\end{cases}.$$

\end{corollary}

(Here $\K_\lambda^*$ denotes the right $S^-(n,r)$-dual module of $\K_\lambda$.)

\section{General coefficient rings}

\q We shall work with Schur algebras over  general coefficient rings.   We will
use the universal coefficient ring $\Z=\zz\left[ t,t^{-1} \right]$.  First we
consider the Schur algebra $S_{\rat(t),t}(n,r)$ over the field of rational
functions in the parameter $t$. We define $S_{\Z,t}(n,r)$ to be 
$$\{\xi\in S_{\rat(t),t}(n,r)\vert \xi(f)\in \Z \hbox{ for all } f\in A_{\Z,t}(n,r)\}$$
which, by Lemma~3.1, is a  $\Z$-form of $S_{\rat(t),t}(n,r)$.
 For an arbitrary commutative ring and a unit $q$ in $R$ we define, by base change via the ring homomorphism  from $\Z$ to $R$, taking $t$ to $q$, the $R$-algebra
$$S_{R,q}(n,r)=R\otimes_{\Z}  S_{\Z,t}(n,r).$$

It is easy to check that, for $R$ a field and $q$ a unit in $R$ this is
consistant with our earlier  definition, i.e., that the homomorphism $\Z\to R$, taking $t$ to $q$ induces an isomorphism $A_{R,q}(n,r)^*\to R\otimes_{\Z} S_{\Z,t}(n,r)$.

\q In the same way we define the negative (quantised) Borel-Schur subalgebra $S^-_{R,q}(n,r)$ of $S_{R,q}(n,r)$. We define $S^-_{\que(t),t}(n,r)=\barA_{\que(t),t}(n,r)^*$. The coalgebra $\overline{A}_{\que(t),t}(n,r)$ has a $\Z$-form $\overline{A}_{\Z}(n,r)$ spanned as a $\Z$-module by the elements $\barc_{ij,\que(t),t}$,  with $i,j\in I(n,r)$. We define 
$S^-_{\Z,t}(n,r)$ to be
$$\{\xi\in S^-_{\rat(t),t}(n,r)\vert \xi(f)\in \Z \hbox{ for all } f\in \overline{A}_{\Z,t}(n,r)\}$$
which, by Lemma~3.3 is a $\Z$-form of   $S^-_{\que(t),t}(n,r)$.  For an
arbitrary commutative ring and a unit $q$ in $R$ we define, by base change via
the ring homomorphism from $\Z$ to $R$ taking $t$ to $q$, the $R$-algebra
$$S^-_{R,q}(n,r)=R\otimes_{\Z}  S^-_{\Z,t}(n,r).$$
It is easy to check that, for $R$ a field and $q$ a unit in $R$ this is
consistent with our earlier  definition, i.e., that the homomorphism $\Z\to R$, taking $t$ to $q$ induces an isomorphism $\overline{A}_{R,q}(n,r)^*\to R\otimes_{\Z} S^-_{\Z,t}(n,r)$.

\q The positive Borel-Schur algebra $S^+_{R,q}(n,r)$ is defined in an analogous
way. We define $A^+_{\que(t),t}(n)=A_{\que(t),t}(n)/I$, where $I$ is the ideal
of $A_{\que(t),t}(n)$ generated by the elements  $c_{ij}$ with $1\le j<i\le n$.
Then $A^+_{\que(t),t}(n)$ has a natural coalgebra grading
\begin{equation*}
	A^+_{\que(t),t}(n)=\bigoplus_{r\geq 0} A^+_{\que(t),t}(n,r).
\end{equation*}
	For every nonnegative $r$ we
define $S^+_{\que(t),t}(n,r)$ to be the $\que(t)$-algebra dual of
$A^+_{\que(t),t}(n,r)$.  We write $A^+_{\Z,t}(n,r)$ for the image of $A_{\Z,t}(n,r)$ under the natural map 
 $A_{\que(t),t}(n,r)\to A^+_{\que(t),t}(n,r)$. Then $A^+_{\Z,t}(n,r)$  is a
 $\Z$-form of $A^+_{\que(t),t}(n,r)$ and we define $S^+_{\Z,t}(n,r)$ to be 
$$\{\xi\in S^+_{\rat(t),t}(n,r)\vert \xi(f)\in\Z  \hbox{ for all } f\in A^+_{\Z,t}(n,r)\}.$$
For an arbitrary commutative ring and a unit $q$ in $R$ we define, by base
change via the ring homomorphism from $\Z$ to $R$ taking $t$ to $q$, the $R$-algebra
$$S^+_{R,q}(n,r)=R\otimes_{\Z}  S^+_{\Z,t}(n,r).$$

We identify $S^-_{R,q}(n,r)$ and $S^+_{R,q}(n,r)$ with $R$-subalgebras of $S_{R,q}(n,r)$ in the obvious way.

\bigskip
\bs

\q We now generalise Corollary~7.10 to   an arbitrary commutative  ground ring from  a general result.   This is presumably well known but we include it here since we were unable to find a suitable reference. For an algebra $S$ over a commutative ring $R$ and maximal ideal $M$ of $R$ with residue field $\K=R/M$ we write $S_\K$ for the $\K$-algebra $\K\otimes_R S$ obtained by base change.  Further, if $D$ is a left (resp. right) $S$-module we write $D_\K$ for the left (\resp right) $S_\K$-module $\K\otimes_R D$ obtained by base change.

\begin{proposition} Let $R$ be a commutative Noetherian ring. Let $S$ be an $R$-algebra which is finitely generated and projective as an  $R$-module. Let $D$ be a right $S$-module and $E$ a left $S$-module. Suppose that $D$ and $E$ are finitely generated and projective as $R$-modules. Suppose further that for each  maximal ideal $M$ of $R$ we have
$$\Tor_i^{S_\K}(D_\K, E_\K)=0$$
for all $i>0$ (where $\K=R/M$). Then  we have
$$\Tor_i^S(D,E)=0$$
 for all $i>0$.

\end{proposition}

\begin{proof} We first  make a reduction to the case in which $R$ is local.  So we first  assume the result in the local case. Let $M$ be a maximal ideal of $R$ and $\K=R/M$. Then we have the $R_M$ algebra $S_M$ obtained by localising at $M$. The $R_M$ module $D_M$ (\resp $E_M$) obtained by localisation is naturally a left (\resp right) $S_M$-module. Also, for $i\geq 0$, we have the localisation $\Tor^S_i(D,E)_M$ of the $R$-module $\Tor^S_i(D,E)$. Moreover, by (the argument of) \cite{Mat}, (3.E), we have
	\begin{equation}
		\label{torsidem}
		\Tor^S_i(D,E)_M\cong \Tor_i^{S_M}(D_M,E_M) 
	\end{equation}
		and 
$$\Tor^{\K \otimes_R S_M}_i(\K \otimes _R D_M,\K \otimes E_M)\cong \Tor_i^{S_\K}(D_\K,E_\K)=0$$
for $i>  0$.   Thus, for $i>  0$, we get $\Tor_i^S(D,E)_M=0$ for all maximal ideals. Since $\Tor_i^S(D,E)$ is a finitely generated $R$-module, this implies $\Tor_i^S(D,E)=0$.

\q We now assume that $R$ is local with maximal ideal $M$ and $\K=R/M$. We make
a reduction to the case $i=1$. Suppose that $\Tor_1^S(D,E)$ is zero for all $D,E$ as above but that the result is  false. We choose $i>1$ as small as possible such that $\Tor_i^S(D,E)\neq 0$ for some $D,E$ as above. We choose an epimorphism from a finitely generated projective $S$-module $P$  onto $E$ and consider the corresponding  short exact sequence of $S$-modules
$$0\to N\to P\to E\to 0.$$
Then $N$ is finitely generated and projective as an $R$-module. Hence we have a short exact sequence of $S_\K$-modules 
$$0\to N_\K\to P_\K\to E_\K \to 0$$
with $P_\K$ projective. Hence we have 

$$\Tor_j^S(D,N)=\Tor^S_{j+1}(D,E) \hbox{\ \ and\ \ } \Tor_j^{S_\K}(D_\K,N_\K)=\Tor_{j+1}^{S_\K}(D_\K,E_\K)$$
for $j\geq 1$. So by the minimality of $i$ we have $\Tor_{i-1}^S(D,N)=0$ and therefore also $\Tor_i^S(D,E)=0$, a contradiction.

 \q Hence it suffices to prove that $\Tor^S_1(D,E)=0$ for all $D,E$ satisfying the hypotheses.  We now consider the right exact functor $\F$ from the category of finitely generated $S$-modules to $\K$-spaces, $\F(X)=X_\K\otimes_{S_\K} E_\K$.  Note that $\F$ factorizes: $\F$ is isomorphic to $\G\circ \H$, where $\H$ is a functor from $S$-modules to $S_\K$-modules, $\H(X)=X_\K$ and $G$ is a functor from the category of $S_\K$-modules to $\K$-spaces $\G(Y)=Y\otimes_{S_\K} E_\K$. Moreover, the functors $\G$ and $\H$ are right exact and $\H$ takes projective $S$-modules to projective $S_\K$-modules. Hence, for $X\in \mod(S)$,  there is a Grothendieck spectral sequence with second page $(L_i\G\circ L_j \H)X$ converging to $(L_*\F)X$. Taking $X=D$, since $D$ is projective as an $R$-module we have $(L_j \H)D=0$ for all $j>0$. Hence the spectral sequence degenerates and we have $(L_i\F)D=(L_i \G)(\H(D))$ for all $i\geq 0$. Hence we have $(L_i \F)D=\Tor^{S_\K}_i(D_\K,E_\K)$, and from the hypotheses, $(L_i\F)D=0$ for all $i>0$. 
 
 \q But also, for a right $S$-module $X$ we have $\F(X)=X_\K\otimes_{S_\K} E_\K = \K\otimes_R (X\otimes_S E)$. This gives another factorisation: $\F$ is the composite $\P\circ \Q$, where $\Q$ is the functor from right $S$-modules to $R$-modules, $\Q(X)=X\otimes_S E$ and $\P$ is the functor from $R$-modules to $\K$-spaces $\P(Y)=\K\otimes_R Y$. For $X$ projective, $X\otimes_S E$ is a projective $R$-module. Hence, for $X$ a right $S$-module,  there is a Grothendieck spectral sequence, with second page $(L_i\P\circ L_j\Q)X$ converging to $(L_*\F)X$. In particular (see \cite{Weibel}, Corollary 5.8.4), we have the $5$-term exact sequence  
 \begin{align*}
  (L_2\F)X\to (L_2 \P)(\Q(X)) &\to \P(L_1\Q(X))\\& \to (L_1\F)X\to (L_1\P)\Q(X)\to 0.
  \end{align*}
	 Taking $X=D$ we obtain the exact sequence
 \begin{align*}
 \Tor^{S_\K}_2(D_\K,E_\K)\to &\Tor_2^R(\K,D\otimes_S E)\to \K\otimes_R \Tor_1^S(D,E)\cr
 \to&\Tor^{S_\K}_1(D_\K,E_\K)\to \Tor_1^R(\K,D\otimes_S E)\to 0.
 \end{align*}
 But $\Tor_i^{S_\K}(D_\K,E_\K)=0$, for $i>0$, and so $\Tor_1^R(\K,D\otimes_S E)=0$. Hence $D\otimes_R E$ is a projective $R$-module, see \cite{Mat}, Section 18, Lemma 4. Hence $\Tor_2^R(\K,D\otimes_S E)=0$ and hence $\K\otimes_R \Tor_1^S(D,E)=0$, and hence
  $\Tor_1^S(D,E)=0$.

\end{proof}

\q Let $R$ be a commutative ring with Noetherian subring $R_0$. Let  $\phi\colon X\to
Y$ be an $R_0$-module homomorphism. It is easy to check (and we leave this to
the reader) that if for every subring  $R'$ of $R$ containing $R_0$ which is
finitely generated over $R_0$, the $R'$-module homomorphism $\phi_{R'}\colon
R'\otimes_{R_0} X\to R'\otimes_{R_0} Y$ is injective then the $R$-module
homomorphism $\phi_R\colon X_R\to Y_R$  (obtained by base change) is injective.

\begin{lemma} Let $R$ be a commutative ring and let $R_0$ be a Noetherian
	subring. Let $S$ be an $R_0$-algebra, finitely generated and projective
	as an $R_0$-module. Let $M$ be a right $S$-module and $N$ a left
	$S$-module and suppose that $M$ and $N$ are finitely generated and
	projective over $R_0$.    If $\Tor^{S_{R'}}_i(M_{R'},N_{R'})=0$ for all
	$i>0$, and all subrings $R'$ of $R$ containing $R_0$ and  finitely generated over $R_0$, then $\Tor^{S_R}_i(M_R,N_R)=0$ for all $i>0$.

\end{lemma}

\begin{proof} Choose an $S$-module surjection $P\to N$, where $P$ is a finitely projective $S$-module and let 
	\begin{equation}
		\label{hpn}
		0\to H\to P\to N\to 0
	\end{equation}
		be the corresponding short exact sequence. Then we have that \\
 $M_{R'}\otimes_{S_R'} N_{R'} \to H_{R'}\otimes_{S_{R'}}  P_{R'}$, is injective, i.e., 
 $$R'\otimes_{R_0} (M\otimes_{R_0} H) \to R'\otimes_{R_0} (M\otimes_{R_0} P) $$
     is injective (whenever $R'$ is a subring of $R$ finitely generated over $R_0$, since $\Tor_1^{S_{R'}}(M_{R'},N_{R'})=0$). Hence we have that 
     $$R\otimes_{R_0} (M\otimes_{R_0} H) \to R\otimes_{R_0} (M\otimes_{R_0} P) $$
      is injective, i.e., $M_R \otimes_{S_R} N_R  \to H_R \otimes_{S_{R}}  P_R$
      is injective and therefore $\Tor_1^{S_R}(M_R,N_R)=0$. Now for $i>1$ it
      follows that $\Tor^{S_R}_i(M_R,N_R)=0$ using~\eqref{hpn}  and dimension shifting.
\end{proof}

\q Let $\lambda\in \Lambda(n,r)$. Then we have the one dimensional module $\K_\lambda$ for the quantised Borel subgroup $B(n)$ over $\K=\que(t)$. Thus $\K_\lambda$ is naturally a left  $S^-_{\que(t),t}(n,r)$-module and we obtain an $S^-_{R,q}(n,r)$-module $R_\lambda$, free of rank one over $R$, by base change. We write $R_\lambda^*$ for the right $S_{R,q}(n,r)$-module dual of $R_\lambda$.
Similarly we construct an $S^+_{R,q}(n,r)$-module, also denoted  $R_\lambda$, free of rank one over $R$.

\begin{theorem}  Let $R$ be a commutative ring and let $q$ be a unit in $R$.  Let  $\lambda\in \Lambda^+(n,r)$. Then we have 
	$$\Tor^{S^-_{R,q}(n,r)}_i(R_\lambda^*,S_{R,q}(n,r))=0$$
for all $i>0$.

\end{theorem}

\begin{proof} The result for $R$ Noetherian follows from Corollary~7.10 and Proposition~8.1 and the result for general $R$ follows from Lemma~8.2.

\end{proof}

\q We shall also give the  version of this result for the positive Borel-Schur
algebra. Suppose $J$ is an anti-automorphism of a ring $S$ and that~$S$ has
subrings $S^-$ and $S^+$ interchanged by $J$.   Given a right $S^-$-module
(\resp left)  $M$ we write $M^J$ for the same group $M$ regarded as a left
(\resp right) $S^+$-module with action $xm=mJ(x)$ (\resp $mx=J(x)m$), $m\in M$,
$x\in S^+$.  Note that if $M$ is $S$, regarded as a right $S^-$-module via right
multiplication, then $M^J$ is $S$ regarded as a left $S^+$-module via
left multiplication.  Similarly if   $M$ is $S$, regarded as a left  $S^-$-module via
left multiplication, then $M^J$ is $S$ regarded as a right $S^+$-module via right
multiplication.  If $M$ is a right $S^-$-module and $N$ is a left $S^-$-module then we have
$$\Tor_i^{S^-}(M,N)\cong  \Tor_i^{S^+}(N^J,M^J)$$
for all $i\geq 0$. 

\q Recall that, by \cite{SD42}, pg.~82, we have an involutary 
anti-auto\-mor\-phism $J$ of the Schur algebra $S_{\K,q}(n,r)$ over a field $\K$.  For
$i\in I(n,r)$ we write $d(i)$ for the number of pairs $(a,b)$ such that $1\leq
a< b\leq r$ and $i_a<i_b$. Then, for $i,j\in I(n,r)$, we have 
$$c_{ji}(\xi)q^{d(j)}=c_{ij}(J(\xi))q^{d(i)}$$
(see \cite{SD42}, p83) and clearly $J$ is determined by this property.  Taking
$\K=\que(t)$ and $q=t$, it is easy to check that $J$ preserves $S_{\Z,t}(n,r)$
and interchanges $S^-_{\Z,t}(n,r)$ and $S^+_{\Z,t}(n,r)$. Hence $J$  induces,
for a general commutative ring $R$ and unit $q\in R$,  an anti-automorphism,
which we also denote $J$, of $S_{R,q}(n,r)$ which interchanges $S^-_{R,q}(n,r)$
and $S^+_{R,q}(n,r)$. Moreover, for $\lambda\in \Lambda(n,r)$, starting with the
left $S^-_{R,q}(n,r)$-module $R_\lambda$, we have that $(R_\lambda^*)^J$ is the
left $S^+_{R,q}(n,r)$-module, also denoted by $R_\lambda$. Thus from Theorem 8.3 we get our quantised version of Woodcock's Theorem.

\begin{theorem}
	\label{mainfirst}
	Let $R$ be a commutative ring and let $q$ be a unit in $R$.  Let  $\lambda\in \Lambda^+(n,r)$. Then we have 
	$$\Tor^{S^+_{R,q}(n,r)}_i(S_{R,q}(n,r),R_\lambda)=0$$
for all $i>0$.

\end{theorem}
Recall that $A_{\rat(t),t}(n,r)$ has basis $\left\{\, c_{ij} \,\middle|\,
(i,j)\in Y(n,r)
\right\}$. Let 
\begin{equation*}
 \left\{\, \xi_{ij} \,\middle|\, (i,j)\in I(n,r) \right\}
 \end{equation*}
	be the dual basis of $S_{\rat(t),t}(n,r)$, that is 
\begin{equation*}
	\xi_{ij}(c_{i',j'}) = 
	\begin{cases}
		1, & i=i',\ j=j'\\
		0, & (i',j') \in Y(n,r),\ (i,j)\not=(i',j').
	\end{cases}
\end{equation*}
Then it is straightforward that $\left\{\, \xi_{ij} \,\middle|\, (i,j)\in
Y(n,r) \right\}$ is a $\Z$-basis of $S_{\Z,t}(n,r)$. We will also denote by the
same symbol $\xi_{ij}$ the image of $\xi_{ij}$ in $S_{R,q}(n,r)$ under base
change. 
For $\lambda\in \Lambda(n,r)$,
we will write $\xi_\lambda$ for $\xi_{l(\lambda), l(\lambda)}$.

	Note that in Section~2 of~\cite{SD42} there is used a slightly different
	parameterisation of the set $\left\{\, \xi_{ij} \,\middle|\, (i,j)\in
	Y(n)
	\right\}$. As we will refer the results of \cite{SD42}, 
	we will 
	explain this in more detail.  Let $U$ be the subset of $I(n,r)\times
	I(n,r)$ of pairs $(l\left( \lambda \right), j)$ such that 
	$j_{1}\ge \dots\ge j_{\lambda_1}$, $j_{\lambda_1 + 1}\ge \dots \ge
	j_{\lambda_1+\lambda_2 } $ and so on. 
Now for every $\lambda\in \Lambda(n,r)$, there is 
a permutation $\pi_\lambda\in \Sigma_\lambda$ of order~$2$ such that
$\left( l(\lambda),j \right) \in Y(n,r)$ if and only if  $\left( l\left( \lambda
\right), j\pi_\lambda
\right)\in U$. 
Since the generators $c_{ab}$ and $c_{ab'}$ commute for any $b$ and $b'$, we
have $c_{l(\lambda),j} = c_{l(\lambda),j\pi_\lambda}$.
Thus we get for any $(i,j)\in Y(n,r)$ and $(i',j')\in U$ that
\begin{equation*}
	\xi_{ij}(c_{i',j'}) = \xi_{ij}\left( c_{i',j'\pi_\lambda} \right)=
	\begin{cases}
		1, & i=i',\ j = j'\pi_\lambda\\
		0, & \mbox{otherwise.}
	\end{cases}
\end{equation*}
Therefore $\xi_{ij}$ in our notation corresponds to $\xi_{i,j\pi_\lambda}$ in the
notation of~\cite{SD42}.

Using the above identification,
from \cite{SD42}, page 38, for $\lambda$, $\mu\in \Lambda(n,r)$ we obtain
\begin{equation*}
	\xi_\lambda \xi_{ij} \xi_\mu=
	\begin{cases}
		\xi_{ij}, & i\in \lambda,\ j\in \mu,\\
		0, &\mbox{otherwise}.
	\end{cases}
\end{equation*}
Moreover, $1= \sum_{\lambda\in \Lambda(n,r)}\xi_\lambda$ is an orthogonal
idempotent decomposition of the identity. 

Similarly to Lemma~3.2, we have that the kernel of the projection
\begin{equation*}
f\colon 	A_{\rat(t),t}(n,r) \twoheadrightarrow A^+_{\rat(t),t}(n,r)
\end{equation*}
has  basis 
\begin{equation*}
	\left\{\, c_{ij} \,\middle|\, (i,j)\in Y(n,r)\mbox{ but not } i\le j
	\right\}.
\end{equation*}
Thus for any $(i,j)\in Y(n,r)$ such that $i\le j$, we get that the restriction
of $\xi_{ij}$ to $\mathrm{Ker}(f) $ is zero. Therefore we can consider
$\xi_{ij}$ as an element
of $S^+_{\rat(t),t}(n,r) = A^+_{\rat(t),t}(n,r)^*$. Using a dimension argument 
we get that 
\begin{equation}
	\label{basis}
	\left\{\, \xi_{ij} \,\middle|\,  (i,j)\in Y(n,r),\ i\le j \right\}
\end{equation}
is a $\rat(t)$-basis of $S^+_{\rat(t),t}(n,r)$. Obviously, it is also a
$\Z$-basis of 
$S^+_{\Z,t}(n,r)$ and, by base change, an $R$-basis for any $S^+_{R,q}(n,r)$. 

We now recall the normalised bar construction.
Note that this is a special case of the construction described in Chapter~IX,~\S 7
of~\cite{homology} and its detailed treatment can be found in Section~3
of~\cite{advances}.  

Let $S$ be a ring with identity and $S'$ a subring of $S$. We assume that
there is an  epimorphism of rings $p\colon S\to S'$ that splits the natural
inclusion of $S'$ into $S$. Write $\kerp$ for the kernel of $p$.  
Then $\kerp$ is an $S'$-bimodule. 

For every $S$-module $M$ we define the chain complex 
\begin{equation*}
B_*\left(
S,S',M
\right) = (B_k\left(
S,S',M \right), d_k)_{k\ge -1}\end{equation*}
 as follows:
\begin{align*}
&	\begin{aligned}	B_{-1}\left( S,S',M \right) & = M, & B_0 \left(S,S',M \right)&= S\otimes
		M,\end{aligned}\\[1ex] & B_k\left( S,S',M \right)  = S \otimes \kerp^{\otimes k}\otimes M,
	\ \forall k\ge 1,\\
& d_k = \sum_{t=0}^k (-1)^t d_{kt}\colon B_{k}\left( S,S',M \right)\to B_{k-1}\left( S,S',M
\right),
\end{align*}
where all the tensor products  are over $S'$ and  
the $S$-module homomorphisms $d_{kt}\colon B_k\left( S,S',M \right)\to
B_{k-1}\left( S,S',M \right)$, $k\ge 0$, $0\le t\le k$ are given
by 
\begin{equation*}
	\begin{aligned}
	&	d_{00} \left( s\otimes m \right)  = sm\\
	&	d_{k0}\left( s\otimes s_1\otimes \dots \otimes s_k\otimes m \
		\right) 
		= ss_1 \otimes s_2 \otimes \dots \otimes s_k\otimes m\\
	&	d_{kt}\left( s\otimes s_1\otimes \dots \otimes s_k \otimes m
		\right) 
		= s \otimes \dots \otimes s_ts_{t+1}\otimes \dots \otimes m,\
		1\le t\le
		k-1\\
	&	d_{kk} \left( s\otimes s_1\otimes\dots \otimes s_k\otimes m
		\right) = s\otimes s_1
		\otimes \dots \otimes s_{k-1}\otimes s_k m,\ k\ge 0. 
	\end{aligned}
\end{equation*}
	The complex $\left( B\left( S,S',M \right), d \right)$ is exact and
is called the \emph{normalised bar resolution} of $M$ over $S$. 
Now we specialize this construction to the case of the quantised Borel-Schur algebra. 

Define
\begin{equation*}
	L = L_{R,q} =  \bigoplus_{(i,i)\in Y(n,r)} R \xi_{ii} =
	\bigoplus_{\lambda\in \Lambda(n,r)} R\xi_{\lambda}
\end{equation*}
 and 
\begin{equation*}
	\ideal = \ideal_{R,q} =  
	\bigoplus_{
	\begin{smallmatrix}
		(i,j)\in Y(n,r)\\[0.5ex]
		i<j
	\end{smallmatrix}
	}R\xi_{ij}.
\end{equation*}
Then $L \oplus \ideal = S^+_{R,q}(n,r)$. 

\begin{proposition}
	The $R$-module $L_{R,q}$ is a split subalgebra of $S^+_{R,q}(n,r)$ and
	$\ideal_{R,q}$ is a split ideal of $S^+_{R,q}(n,r)$.
\end{proposition}
\begin{proof}
	It is obvious that $L_{R,q}$ is a subalgebra of $S^+_{R,q}(n,r)$.
	
	Now, we will check that $\ideal_{R,q}$ is an ideal of $S^+_{R,q}(n,r)$. 
	By a base change argument, it is enough to check that $\ideal_{\Z,t}$ is
	an ideal of $S^+_{\Z,t}(n,r)$ and this can be reduced to showing that
	$\ideal_{\rat(t),t}(n,r)$ is an ideal of $S^+_{\rat(t),t}(n,r)$.
	
	Let $(i,j)$, $(i',j')\in Y(n,r)$ such that $i\le j$ and $i'\le j'$. Then the coefficient of $\xi_\lambda$
	in the expansion of the product $\xi_{i,j}\xi_{i',j'}$ in  the basis~\eqref{basis}
	of $S^+_{\rat(t),t}(n,r)$
	is given by
	\begin{equation*}
		(\xi_{i,j} \xi_{i',j'})	\left( \overline{ c }_{l\left( \lambda
		\right), l\left( \lambda \right)} \right) = 
		\sum_{h\in I(n,r)} \xi_{i,j}\left( \overline{ c}_{l(\lambda),h}  \right) 
		\xi_{i',j'}\left( \overline{ c }_{h,l\left( \lambda \right)}
		\right),
	\end{equation*}
	where $\overline{ c }_{ij}$ denotes the image of $c_{ij}$ under the
	epimorphism
	\begin{equation*}
		A_{\rat(t),t}(n,r)\twoheadrightarrow A^+_{\rat(t),t}(n,r).
	\end{equation*}
	Thus $\overline{ c }_{l(\lambda),h}\not=0$ and $\overline{ c }_{h,l\left(
	\lambda \right)}\not=0$ imply that $l(\lambda)\le h\le l(\lambda)$.
	Hence 
	\begin{equation*}
			(\xi_{i,j} \xi_{i',j'})	\left( \overline{ c }_{l\left( \lambda
			\right), l\left( \lambda \right)} \right) =
			\xi_{i,j}\left( \overline{ c }_{l(\lambda),l(\lambda)}
			\right) \xi_{i',j'}\left( \overline{ c
			}_{l(\lambda),l(\lambda)} \right),
	\end{equation*}
	which is zero, if $i\not=j$ or $i'\not=j'$. This proves that the product
	$\xi_{i,j}\xi_{i',j'}$ lies in $\ideal_{\rat(t),t}$, if $\xi_{i,j}\in
	\ideal_{\rat(t),t}$
	or $\xi_{i',j'}\in \ideal_{\rat(t),t}$.
\end{proof}
 For any $\lambda\in\Lambda(n,r)$ we can apply the normalised bar construction to
$S^+_{R,q}(n,r)$, ${L}$, and the rank-one module $R_\lambda$.

Denote $B_k\left( S^+_{R,q}(n,r), \ideal, R_\lambda \right)$ by $B^+_{k,\lambda}$ for
$k\ge -1$. 
We get
\begin{align*}
	&B^+_{-1,\lambda}  = R_\lambda, \  
	B^+_{0,\lambda}  = S^+_{R,q}\left( n,r \right)\otimes_{L}
	R_\lambda,\\[2ex]
	& B^+_{k,\lambda}  = S^+_{R,q}\left( n,r \right)\otimes_{L}
	\ideal^{\otimes k} \otimes_{L}
	R_{\lambda},\ k\ge 1.
\end{align*}
For any $\mu\in \Lambda(n,r)$,  $M\in \mbox{mod-}L$ and
$N\in L\mbox{-mod}$, we have
\begin{equation*}
	\left( M\otimes_L R\xi_\mu \right)\otimes_{R} \left(
	R\xi_\mu\otimes_L N
	\right) \cong M\otimes_L R\xi_\mu \otimes_L N.
\end{equation*}
Thus
\begin{align*}
	M\otimes_L N &\cong M\otimes_L L \otimes_L N \cong \bigoplus_{\mu\in
	\Lambda}
	\left( M\otimes_L R\xi_\mu \right) \otimes_{R} \left( R\xi_\mu\otimes_L
	N
	\right) \\& \cong \bigoplus_{\mu\in\Lambda} M\xi_{\mu}\otimes_R \xi_\mu N,
\end{align*}
since $M\otimes_L R\xi_\mu\cong M\xi_\mu$ and $R\xi_\mu\otimes_L N \cong \xi_\mu
N$. Hence
\begin{align*}
	B^+_{0,\lambda}&\cong 
	\bigoplus_{\mu\in \Lambda} S^+_{R,q}\left( n,r \right)\xi_{\mu}\otimes_R
	\xi_{\mu} R_\lambda \\ &= 
	S^+_{R,q}\left( n,r \right) \xi_\lambda \otimes_R R_\lambda \cong
	S^+_{R,q}\left( n,r \right) \xi_\lambda,
\end{align*}
since $\xi_\mu R_\lambda = 0$ unless $\mu = \lambda$. 
Further
\begin{multline*}
	B_{k,\lambda}^+
	\cong \bigoplus_{\mu^{\left( 1 \right)},\dots,\mu^{\left( k+1 \right)}\in \Lambda}
	S^+_{R,q}\left( n,r \right)\xi_{\mu^{( 1 )}}\otimes_R
	\xi_{\mu^{\left( 1 \right)}} \ideal \xi_{\mu^{\left( 2 \right)}} \otimes_R
	\dots\\
	\otimes_R \xi_{\mu^{\left( k \right)}} \ideal \xi_{\mu^{\left( k+1
	\right)}} \otimes_R \xi_{\mu^{\left( k+1 \right)}} R_\lambda.
\end{multline*}
As $\left\{\, \xi_{ij} \,\middle|\,  \left( i,j \right)\in
Y\left( n,r
\right), \ i<j\right\}$ is an $R$-basis of $\ideal$, we get that $\xi_\mu
\ideal\xi_\tau$ is zero,
unless $\mu\gtdom \tau$. 
If $\mu\gtdom \tau$, then $\xi_{\mu} \ideal \xi_\tau$ has an $R$-basis 
\begin{equation*}
	\left\{\,  \xi_{ij} \,\middle|\,  (i,j)\in Y\left( n,r \right),\
	i<j,\ j\in \tau, i\in \mu
	\right\}. 
\end{equation*}
Thus for every $k\ge 1$ we can write 
\begin{equation*}
	B^+_{k,\lambda} \cong
	\bigoplus_{\mu^{(1)}\gtdom \dots\gtdom \mu^{(k)}\gtdom \lambda} 
	S^+_{R,q}\left( n,r \right)\xi_{\mu^{(1)}} \otimes
	\xi_{\mu^{(1)}} \ideal \xi_{\mu^{(2)}} \otimes \dots 
	\otimes \xi_{\mu^{(k)}} \ideal \xi_\lambda,
\end{equation*}
where $\otimes$ means $\otimes_R$. 
Note that $B_{k,\lambda}^+$ is zero for $k$ sufficiently large.

Given $\mu\gtdom \lambda$, we define $\Omega^+_k\left( \lambda,\mu \right)$ to
be the set of all sequences 
\begin{equation*}
(i^{(1)}, j^{(1)}),\dots,(i^{(k)}, j^{(k)}) 
\end{equation*}
of elements in $Y(n,r)$ such that $i^{(1)}\in \mu$, $j^{(k)}\in \lambda$,
and 
\begin{equation*}
	i^{(1)}<j^{(1)}\sim i^{(2)} <j^{(2)} \sim i^{(3)}< \dots <j^{(k)},
\end{equation*}
where $j\sim i$ means that $i$ and $j$ have the same content.  
Then  we have isomorphisms of $S^+_{R,q}(n,r)$-modules
\begin{equation}
	\label{eq:bk}
	B^+_{k,\lambda} \cong
	\begin{cases}
		S^+_{R,q}(n,r)\xi_\lambda, & k = 0 \\
\bigoplus_{\mu\gtdom \lambda }
\left(
S^+_{R,q}\left( n,r \right)\xi_{\mu} \right)^{ \#\Omega^+_k\left(\lambda, \mu
\right)}, & k\ge 1 .
	\end{cases}
\end{equation}
So we  get the following result
\begin{theorem}
 \label{bpluslambda}
 Let $\lambda\in \Lambda(n,r)$. Then the complex $B^+_{*,\lambda}$ is a projective
 resolution of $R_\lambda$ over $S^+_{R,q}(n,r)$. 
\end{theorem}
Now consider $\lambda\in\Lambda^+(n,r)$.
The \emph{Weyl module}  associated with $\lambda$ is 
\begin{equation*}
W_\lambda=	S_{R,q}(n,r) \otimes_{S^+_{R,q}(n,r)}R_\lambda.
\end{equation*}
By
Theorem~\ref{mainfirst}, $R_\lambda$ is an acyclic module for the functor
$S_{R,q}\otimes_{S^+_{R,q}(n,r)} -$. Therefore 
	$B_{*,\lambda}:=
	S_{R,q}(n,r)\otimes_{S^+_{R,q}(n,r)}
	B_{*,\lambda}^+$
	is a projective resolution of~$W_\lambda$. Moreover, since
\begin{equation*}
	S_{R,q}(n,r)\otimes_{S^+_{R,q}(n,r)} S^+_{R,q}(n,r) \cong S_{R,q}(n,r),
\end{equation*}
	we
get
the following theorem. 
\begin{theorem}
	\label{weylresolution}
	Let $\lambda\in \Lambda^+(n,r)$. Define the complex $B_{\lambda}$ as
	follows:
	\begin{align*}
		B_{-1,\lambda}&= W_\lambda, &	B_{0,\lambda} & =
		S_{R,q}(n,r)\xi_\lambda ,
	\end{align*}
	and for $k\ge 1$, we set
	$B_{k,\lambda}$ to be
	\begin{align*}
   \bigoplus_{
    \begin{smallmatrix}
      	 \mu^{(1)}\gtdom \dots \gtdom \mu^{(k)}\gtdom \lambda\\
  \mu^{(1)}, \dots , \mu^{(k)}\in \Lambda(n,r)
   \end{smallmatrix}
    }
    	S_{R,q}\left( n,r \right)\xi_{\mu^{( 1 )}}\otimes_R
	\xi_{\mu^{\left( 1 \right)}} \ideal \xi_{\mu^{\left( 2 \right)}} \otimes_R
	\dots
	\otimes_R \xi_{\mu^{\left( k \right)}} \ideal \xi_{\lambda}.
\end{align*}
Define $d_0$ to be the canonical projection of $S_{R,q}(n,r)\xi_\lambda$ on
$W_\lambda$, and
for $k\ge 1$ define $d_k \colon B_{k,\lambda}\to B_{k-1,\lambda}$ to be the
$R$-linear extension of the map
\begin{align*}
	x_0 \otimes x_1 \dots \otimes x_k & \mapsto \sum_{t=0}^{k-1}
(-1)^t	x_0 \otimes \dots \otimes x_t x_{t+1} \otimes \dots \otimes x_k.
\end{align*}
Then $B_{\lambda}$ is a projective resolution of $W_\lambda$ over
$S_{R,q}(n,r)$. 
\end{theorem}
We will show now that  $B_\lambda$ is  stable under base change.
Our resolutions for a moment will get   additional indices to emphasize dependence on
$R$ and $q\in R$. 
Let $R$ and $R'$ be commutative rings, $\theta\colon R\to R'$ a
ring homomorphism, $q\in R$ and $q':= \theta(q)\in R'$ 
invertible elements. 
 Since $S_{R,q}(n,r)\otimes_R
R' \cong S_{R',q}(n,r)$ and $B^R_{k,\lambda}$ are free $S_{R,q}(n,r)$-modules
for $k\ge 0$, we get that $( B^R_{k,\lambda}\otimes_R R'
,\ k\ge 0 )$  and $( B^{R'}_{k,\lambda},\ k\ge 0 )$ are
isomorphic complexes. 
Moreover, from the commutative diagram with exact rows
\begin{equation*}
	\xymatrix{
	B^R_{1,\lambda}\otimes_{R} R' \ar[r]^{d_{1}} \ar[d]_{\cong} &
	B^R_{0,\lambda}\otimes_R R'
 \ar[r]^{d_0} \ar[d]_{\cong} & W^R_\lambda\otimes_R R'\ar[r]
 \ar@{-->}[d]^{\exists!}_{\cong}& 0 \\
  B^{R'}_{1,\lambda}\ar[r]^{d_1} & B^{R'}_{0,\lambda}\ar[r]^{d_0} &
  W^{R'}_\lambda\ar[r] & 0,
 }
\end{equation*}
it follows that $B^R_{*,\lambda}\otimes_R R'$ and $B^{R'}_{*,\lambda}$ are
isomorphic also in degree $-1$. 
\section{The Hecke algebra and resolutions of co-Specht modules}
In this section we will use the notation of~\cite{dipperjames}
but  will denote by $\lng\left( \sigma \right)$ the length of  $\sigma \in \Sigma_r$. 
The \emph{Hecke algebra} $\hecke = \hecke_{R,q}$ associated with $\Sigma_r$ over
$R$ is free as an $R$-module with basis $\left\{\, T_\sigma \,\middle|\,
\sigma\in \Sigma_r
\right\}$, where
\begin{equation*}
T_s T_\sigma =
\begin{cases}
	T_{s\sigma}, & \mbox{if $\lng(s\sigma) = \lng(\sigma) + 1$}\\
	qT_{s\sigma} + (q-1)T_{\sigma}, & \mbox{otherwise},
\end{cases}
\end{equation*}
for $\sigma$, $s\in \Sigma_r$ with $\lng(s) = 1$.

In \cite{boltjemaish} Boltje and Maisch constructed for every
composition $\lambda$ of $r$ a chain complex
$\widetilde{C}_*^\lambda$ of $\hecke$-modules. These complexes
are lifting to the $q$-setting of the corresponding
$R\Sigma_r$-module complexes described in~\cite{boltjehartmann}. It
was proved in \cite{advances}, that $\widetilde{C}^\lambda_*$ is a 
permutation resolution of the co-Specht modules $\Hom_R(S^\lambda,R)$ for $q=1$ and
$\lambda$ a partition  of $r$. In this section we will prove a similar
result for any invertible $q$ in $R$.  

Choose any $n\ge r$, and let 
\begin{align*}
&\omega = (1,\dots,1,0,\dots, 0)\in
\Lambda(n,r)\\ &u = (1,2,\dots, r)\in I(n,r).
\end{align*}
	Then (see \cite[Section~0.23]{SD42})
\begin{equation*}
	\left\{\, \xi_{u,u\pi} \,\middle|\,  \pi\in \Sigma_r \right\}
\end{equation*}
is an $R$-basis of $\xi_\omega S_{R,q}(n,r) \xi_\omega$ and
\begin{align*}
	\xi_\omega S_{R,q}(n,r) \xi_\omega & \to \hecke\\
	\xi_{u,u\pi} & \mapsto T_{\pi^{-1}}
\end{align*}
is an isomorphism of $R$-algebras. Therefore we have the Schur functor
\begin{align*}
	F \colon S_{R,q}(n,r)\mbox{-mod} & \to \hecke\mbox{-mod}\\
	V \mapsto \xi_{\omega} V.
\end{align*}
Applying $F$ to the resolution $B_\lambda$ of $ W_\lambda$ with $\lambda\in
\Lambda^+(n,r)$,  we obtain an 
exact sequence $F\left( B_\lambda \right)$. It is our aim to prove
that $F\left( B_\lambda \right)$ and $\widetilde{C}^\lambda_*$ are
isomorphic chain complexes of $\hecke$-modules. This
will prove that the complexes $\widetilde{C}^\lambda_*$ are 
resolutions of the co-Specht modules $\Hom_R(S^\lambda,R)$ over $\hecke$.

We start by reminding to the reader some facts on Hecke algebra.
Denote by $\Sigma_\lambda$ the
standard Young subgroup corresponding to the composition~$\lambda$. 
By \cite[Lemma~1.1]{dipperjames} each right coset of $\Sigma_\lambda$ in
$\Sigma_r$ contains a unique element of minimal length, the \emph{distinguished}
coset representative of $\Sigma_\lambda$ in $\Sigma_r$. We denote by
$D_\lambda$ the set of these elements.  Given two compositions $\lambda$ and
$\mu$, we
also define $D_{\lambda,\mu} = D_\lambda \cap D_{\mu}^{-1}$. By
\cite[Lemma~1.6]{dipperjames} the set $D_{\lambda,\mu}$ is a system of 
$\Sigma_\lambda$-$\Sigma_\mu$ double coset representatives in $\Sigma_r$.

Recall that, for $\lambda\in \Lambda(n,r)$,  we write $l(\lambda)$ for the multi-index 
$(1^{\lambda_1}, 2^{\lambda_2}, \dots, n^{\lambda_n})$. Then every element of $Y(n,r)$ is  of the
form $(l\left( \lambda \right), j)$ for some $\lambda\in \Lambda(n,r)$ and
$j\in I(n,r)$. It is easy to see (cf. \cite[pp.~24-25]{SD25}), that for given
$\lambda$, $\mu\in \Lambda(n,r)$, there is a bijective correspondence 
\begin{equation}
	\label{7.1}
	\left\{\, (l(\lambda),j)\in Y(n,r) \,\middle|\, j\in \mu \right\}
	 \to D_{\lambda,\mu}, 
\end{equation}
defined as follows. For a given pair $(l(\lambda), j )$ the set
\begin{equation*}
\left\{\, \pi\in
\Sigma_r
\,\middle|\, l(\mu)\pi^{-1} = j \right\}
\end{equation*}
	is a $\Sigma_\mu$-orbit, and thus
contains a unique distinguished element $\bar{d}$ of $D_{\mu}^{-1}$. We define
$d$ as 
the representative of $\Sigma_\lambda \overline{ d } \Sigma_\mu$ in
$D_{\lambda,\mu}$. 

For $\lambda\in \Lambda(n,r)$, define $x_\lambda := \sum\limits_{\pi\in
\Sigma_\lambda} T_\pi$  and $M^\lambda := x_\lambda\hecke$. 
Then $\Hom_{\hecke}(M^\mu, M^\lambda)$ has an $R$-basis 
\begin{equation*}
	\left\{\, \varphi^{\lambda,\mu}_d \,\middle|\, d\in D_{\lambda,\mu}
	\right\},
\end{equation*}
where 
\begin{equation*}
	\varphi^{\lambda,\mu}_d (x_\mu) = \sum_{\pi\in \Sigma_\lambda d \Sigma_\mu}
	T_{\pi},\ d\in D_{\lambda,\mu}.
\end{equation*}
Theorem~3.2.5 and Corollary~3.2.6 in~\cite{SD25} say that
there is an algebra isomorphism
\begin{align}
	\label{8.1}
	S_{R,q}(n,r) & \to \bigoplus_{\mu,\lambda\in \Lambda(n,r)}
	\Hom_{\hecke}(M^\mu,M^\lambda)\\[2ex]
	\nonumber
	\xi_{l(\lambda), j} & \mapsto \varphi^{\lambda,\mu}_d,
\end{align}
where the correspondence $(l(\lambda),j)\mapsto d$ is given by \eqref{7.1}. 

Denote by $\tab (\lambda,\mu)$ the set of all $\lambda$-tableaux with content
$\mu$ and by $\tab^{rs}(\lambda,\mu)$ the set of all row semistandard
$\lambda$-tableaux with content $\mu$. 
Write
\begin{equation*}
	T^\lambda = 
	\begin{array}{cccccc}
		1 & 2 & \dots & \lambda_1\\
		\lambda_1 + 1 & \lambda_1 + 2 & \dots & \dots & \dots & \lambda_1 +
		\lambda_2 \\
		 \dots \\
		 \lambda_1 + \dots + \lambda_{n-1} + 1 & \dots & \dots & \dots & r
	\end{array}
\end{equation*}
and for each $i\in I(n,r)$, let $T^\lambda_i$ be the $\lambda$-tableaux
\begin{equation*}
	T_i^\lambda = 
	\begin{array}{cccccc}
		i_1 & i_2 & \dots & i_{\lambda_1}\\
		i_{\lambda_1 + 1} & i_{\lambda_1+2} & \dots & \dots & \dots &
		i_{\lambda_1 + \lambda_2}\\
		\dots \\
		i_{\lambda_1 +\dots \lambda_{n-1} + 1} & \dots & \dots & \dots &
		i_r.
	\end{array}
\end{equation*}
Recall that $(i,j)\in Y(n,r)$ if and only if $i_1\le i_2 \le \dots \le
i_r$ and $j_{\nu}\le j_{\nu+1}$ if $i_\nu = i_{\nu+1}$, $1\le \nu \le r-1$. 
Therefore there is a bijective correspondence
\begin{align*}
	\left\{\, \left( l\left( \lambda \right), j \right)\in Y\left(
	n,r \right) \,\middle|\,  j\in \mu \right\} & \to \tab^{rs}\left(
	\lambda, \mu
	\right)\\
	\left( l\left( \lambda \right), j \right) & \mapsto T^\lambda_j
\end{align*}
that in combination with \eqref{7.1} induces the bijection
\begin{equation}
	\label{9.1}
	D_{\lambda,\mu} \leftrightarrow \tab^{rs}\left( \lambda,\mu \right). 
\end{equation} 
Boltje and Maisch say that  a $\lambda$-tableaux in $\tab(\lambda, \mu)$ is
\emph{ascending} if, for every $a\in \mathbb{N}$, the $a$th row of this tableau contains
only entries which are greater than or equal to $a$. They denote the set of all
ascending elements of $\tab^{rs}\left( \lambda,\mu \right)$ by
$\tab^{\wedge}(\lambda, \mu)$. 
One has $\tab^\wedge\left( \lambda,\mu \right)\not=\varnothing$ if and only if
$\mu\ledom \lambda$, if and only if $T^{\lambda}_{l\left( \mu \right)}\in
\tab^\wedge(\lambda,\mu)$. 
Notice that for $j\in I(n,r)$, the $\lambda$-tableau $T^\lambda_j$ is ascending
if and only if $l\left( \lambda \right)\le j$. Therefore we have a bijective
correspondence
\begin{align*}
 Y\left( \lambda,\mu \right)^\wedge := \left\{\, \left( l\left( \lambda
	\right), j \right)\in Y\left(n,r \right) \,\middle|\,  j\in \mu,\ l\left( \lambda \right)\le j
	\right\} & \to \tab^\wedge\left( \lambda,\mu \right)\\
	(l(\lambda), j) & \mapsto T^\lambda_j. 
\end{align*}
Denote by $D^\wedge_{\lambda,\mu}$ the image of $Y(\lambda,\mu)^\wedge $ in
$D_{\lambda,\mu}$ under the correspondence~\eqref{7.1}. Boltje and Maisch define
for each $\mu\ledom \lambda$
\begin{equation*}
	\Hom^\wedge_{\hecke} (M^\mu, M^\lambda) := 
	\bigoplus_{d\in D^\wedge_{\lambda,\mu}} R\varphi^{\lambda,\mu}_d \subset
	\Hom_{\hecke} (M^\mu,M^\lambda). 
\end{equation*}
Then under the isomorphism~\eqref{8.1}, 
$\Hom^\wedge_{\hecke}(M^\mu,M^\lambda)$ corresponds to
\begin{equation*}
 \bigoplus\limits_{(l(\lambda), j)\in
 I^2(\lambda,\mu)^\wedge}R\xi_{l(\lambda),
 j}.\end{equation*}

But, since 
\begin{equation*}
	\left\{\, \xi_{l(\lambda), j} \,\middle|\, (l(\lambda), j) \in
	Y(n,r),\ l(\lambda)\le j,\ \lambda\in \Lambda \right\}
\end{equation*}
is an $R$-basis of $S_{R,q}^+(n,r)$ and for any $\nu$, $\tau\in \Lambda(n,r)$
\begin{equation*}
	\xi_{\nu} \xi_{ij} \xi_{\tau} = 
	\begin{cases}
		\xi_{ij}, & \mbox{if $i\in \nu$, $j\in \tau$}\\
		0 , &\mbox{otherwise,}
	\end{cases}
\end{equation*}
we get that $\Hom^\wedge_{\hecke}(M^\mu,M^\lambda)$ corresponds to
$\xi_\lambda S^+_q(n,r)\xi_\mu$. We saw previously that $S^+_{R,q}(n,r) = L
\oplus \ideal$. But if $\lambda\gtdom \mu$, we have $\xi_\lambda L \xi_\mu = 0$.
Hence
$\Hom^\wedge_{\hecke} (M^\mu, M^\lambda)$ corresponds to $\xi_\lambda \ideal
\xi_\mu$ if $\lambda \gtdom \mu$. 

Next we define the Boltje-Maisch complex $\widetilde\bmcom^\lambda_*$. We will
restrict ourselves to the case
when $\lambda$ is a partition of $r$.
 For every right $\hecke$-module $N$ the
$R$-module $\Hom_R(N,R)$ has the structure of  a left
$\hecke$-module given by 
\begin{equation*}
	(h\varepsilon)(n)  := \varepsilon (nh),
\end{equation*}
where $h\in \hecke$, $\varepsilon\in \Hom_R(N,R)$, and $n\in N$.
So given an $R$-module $N'$, the $R$-module $\Hom_R(N,R)\otimes_R N'$ can
be viewed as an $\hecke$-module via
\begin{equation*}
 h (\varepsilon \otimes n') = (h\varepsilon )\otimes n',
\end{equation*}
where $h\in \hecke$, $\varepsilon \in \Hom_R(N,R)$, and $n'\in
N'$.

For each $\lambda\in \Lambda^+(n,r)$, Boltje and Maisch define a
complex
\begin{equation*}
 \widetilde{C}^\lambda_*\colon
 0 \to C^{\lambda}_{a(\lambda)}
 \stackrel{d^\lambda_{a(\lambda)}}{\longrightarrow}
 C^\lambda_{a(\lambda)-1} 
  \stackrel{d^\lambda_{a(\lambda)-1}}{\longrightarrow}
 \dots 
 \stackrel{d^\lambda_1}{\longrightarrow}
 C_0^\lambda
 \stackrel{d^\lambda_0}{\longrightarrow}
 C^\lambda_{-1}\to 0
\end{equation*}
in the following way:
\begin{equation*}
 \bmcom^\lambda_{-1} = \Hom_{R}(S^\lambda,R), 
\end{equation*}
\begin{multline*}
 \bmcom^\lambda_k = \!\!\!\!
\bigoplus_{
\begin{smallmatrix}
	 \mu^{(1)}\gtdom \dots \gtdom \mu^{(k)}\gtdom \lambda\\
	 \mu^{(1)}, \dots , \mu^{(k)}\in \Lambda(n,r)
	\end{smallmatrix}
	}\!\!\!\!
	\Hom_{R}\left( M^{\mu^{(1)}}, R \right) \otimes_R 
	\Hom^\wedge_{\hecke} (M^{\mu^{(2)}}, M^{\mu^{(1)}})
      \\
	\otimes_R
       \cdots
     \otimes_R 
     \Hom^\wedge_{\hecke} (M^\lambda, M^{\mu^{(k)}}).
    \end{multline*}
The differential $d_k^\lambda\colon \bmcom_k^\lambda\to
\bmcom^\lambda_{k-1}$ is given by the sum
$\sum\limits_{t=0}^{k-1} (-1)^t d_{kt}$, where
for
$k\ge 1$ and $1\le t\le k-1$, we set
\begin{equation}
	\label{differential}
	\begin{aligned}
		d_{k0} (\varepsilon \otimes \phi_1 \otimes \dots\otimes  \phi_k)
	&= 
	\varepsilon \phi_1 \otimes \phi_2 \dots \otimes \phi_k,\\
	d_{kt} (\varepsilon \otimes \phi_1 \otimes \dots\otimes  \phi_k) &= 
	\varepsilon \otimes \phi_1 \otimes \dots \otimes \phi_t \phi_{t+1}\otimes \dots
	\otimes  \phi_k,
\end{aligned}
\end{equation}
	and $d_0\colon \Hom_{R}(M^\lambda, R)\to \Hom_R(S^\lambda,R)$ is
defined to be the restriction on $S^\lambda$.

Let us consider the resolution $B_\lambda$ of $W_\lambda$. Applying the Schur
functor to $B_\lambda$ we obtain the exact sequence $F(B_\lambda)$, where
	\begin{align*}
		F(B_\lambda)_{-1} &= \xi_\omega W_\lambda, & 
		F(B_\lambda)_0 & = \xi_\omega S_{R,q}(n,r) \xi_\lambda,
\end{align*}	
and for $k\ge 1$ the $\hecke$-module $F(B_\lambda)_k$ is given by
\begin{equation*}
\bigoplus_{
\begin{smallmatrix}
	 \mu^{(1)}\gtdom \dots \gtdom \mu^{(k)}\gtdom \lambda\\
 \mu^{(1)}, \dots , \mu^{(k)}\in \Lambda(n,r)
\end{smallmatrix}
}
\xi_\omega S_q(n,r) \xi_{\mu^{(1)}} \otimes_R 
\xi_{\mu^{(1)}} \ideal \xi_{\mu^{(2)}} \otimes_R 
\dots
\otimes_R 
\xi_{\mu^{(k)}} \ideal \xi_\lambda.
	\end{equation*}
Notice that, for $\mu\in\Lambda(n,r)
$
the subspace $\xi_\omega S_q(n,r)\xi_\mu$ corresponds
under~\eqref{8.1} to
$\Hom_{\hecke}(M^\mu, M^\omega)$. But $M^\omega = x_\omega
\hecke = \hecke$, since $\Sigma_\omega$ is the trivial group and
$x_\omega = \sum_{\pi\in \Sigma_\omega} T_\pi=T_\id$. Thus $\xi_\omega
S_{R,q}(n,r) \xi_\mu$ corresponds under~\eqref{8.1} to $\Hom_{\hecke}
(M^\mu, \hecke)$.
 Here we have that $\Hom_{\hecke} (M^\mu,
\hecke)$ is a $\hecke$-module by 
\begin{equation*}
	(h\psi)(m) = h\psi(m),
\end{equation*}
where $h\in \hecke$, $m\in M^\mu$, and $\psi\in \Hom_{\hecke} (M^\mu,
\hecke)$.

Thus we can write
\begin{align*}
	F(B_\lambda)_0 & =  \Hom_{\hecke}(M^{\lambda},
	\hecke),
\end{align*}
\begin{multline*}
F(B_\lambda)_k   = \!\!\!\! 
\bigoplus_{
\begin{smallmatrix}
	 \mu^{(1)}\gtdom \dots \gtdom \mu^{(k)}\gtdom \lambda\\
 \mu^{(1)}, \dots , \mu^{(k)}\in \Lambda(n,r)
\end{smallmatrix}}\!\!\!\!\!\!\!\!\!\!\!\!
\Hom_{\hecke}(M^{\mu^{(1)}}, \hecke)
\otimes_R 
	\Hom^\wedge_{\hecke} (M^{\mu^{(2)}}, M^{\mu^{(1)}})
      \\
	\otimes_R
       \cdots
     \otimes_R 
     \Hom^\wedge_{\hecke} (M^\lambda, M^{\mu^{(k)}}), 
\end{multline*}
and the differentials $d_k$ in $F(B_\lambda)$  
are the sums $\sum_{t=0}^{k-1}(-1)^t d_{kt}$, where the maps
$d_{kt}$ are defined analogously to~\eqref{differential}. 

We will prove that $F(B_\lambda)$ is isomorphic to the complex
$\widetilde\bmcom^\lambda_*$ in
non-negative degrees. Since $F(B_\lambda)$ is exact and
$\widetilde\bmcom^\lambda_*$ is exact in the
degrees $0$ and $-1$ by \cite[Theorems~4.2~and~4.4 ]{boltjemaish}, the isomorphism in degree $-1$ will follow.

To prove that $F(B_\lambda)_k\cong \bmcom^\lambda_k$ for $k\ge
0$, we start by 
showing that there is an isomorphisms of $\hecke$-modules
\begin{equation*}
	\mathfrak{F}_\mu\colon \Hom_{\hecke} (M^\mu, \hecke)  \to 
	\Hom_{R}(M^\mu, R)
\end{equation*}
such that for all $\nu\in \Lambda(n,r)$, $\psi\in \Hom_{\hecke} (M^\mu,
\hecke)$, $\varphi\in \Hom_{\hecke} (M^\nu, M^\mu)$, we have
$\mathfrak{F}_\nu(\psi\varphi) = \mathfrak{F}_\mu(\psi)\varphi$. 

We will prove this in a more general setting.
Let ${}^*\colon \hecke\to \hecke$ be the anti-automorphism of
$\hecke$ given by $T_\pi \mapsto T^*_\pi =T_{\pi^{-1}}$.
Let $M$ be any right
$\hecke$-module. By \cite[Theorem~2.6]{dipperjames} there is an
isomorphism of $R$-modules 
\begin{align*}
	\Hom_R(M,R) & \to \Hom_{\hecke} ( M, \hecke)\\
	\varphi & \mapsto \hat{\varphi},
\end{align*}
where
\begin{equation*}
	\hat{\varphi}(m) := 
	\sum_{\sigma \in \Sigma_r} q^{-\lng(\sigma)}
	\varphi(mT^*_{\sigma})T_\sigma.
\end{equation*}
The inverse of this isomorphism is the map
\begin{align*}
	\Hom_{\hecke} \left( M, \hecke \right) &\to \Hom_{R}(M, R)\\
	\psi & \mapsto \tilde{\psi},
\end{align*}
where $\tilde{\psi}(m)$ is the coefficient of $T_\id$ in the expansion 
\begin{equation*}
	\psi(m) = \sum_{\sigma \in \Sigma_r} a_\sigma T_\sigma, \ a_\sigma\in R.
\end{equation*}
 Consider the
symmetric associative bilinear form $f\colon \hecke\otimes
\hecke \to R$ (\cite[Lemma~2.2 and proof of Theorem~2.3]{dipperjames}) given by 
\begin{equation*}
f(T_\sigma, T_\pi) = 
\begin{cases}
	q^{\lng(\sigma)}, & \mbox{if $\sigma = \pi^{-1}$}\\
	0 , & \mbox{otherwise.}
\end{cases}
\end{equation*}
Note that we have
\begin{equation*}
	f(\sum_{\sigma\in \Sigma_r} a_\sigma T_\sigma, T_\id) = a_\id (T_\id, T_\id) =
	a_\id.
\end{equation*}
Thus for $m\in M$ we get $\tilde{\psi}(m) = f(\psi(m), T_\id)$. We will prove that
$\psi\to \tilde{\psi}$ is an $\hecke$-module homomorphism. Recall that
$\Hom_R(M,R)$ is a left $\hecke$-module by $(h\varphi)(m) =
\varphi(mh)$, where $h\in \hecke$, $m\in M$, $\varphi\in
\Hom_R(M,R)$, and $\Hom_{\hecke}(M,\hecke)$ is a left
$\hecke$-module by $(h\psi)(m) = h\psi(m)$, where $h\in \hecke$,
$\psi\in \Hom_{\hecke}(M,\hecke)$, $m\in M$. 
\begin{proposition}
	\label{second1}
	The map
	\begin{equation}
			\label{ffrak}	
			\begin{aligned}
					\mathfrak{F}_M\colon
					\Hom_{\hecke}(M,\hecke) & \to
					\Hom_R(M,R)\\
						\psi & \mapsto \tilde{\psi}
						\end{aligned}
					\end{equation}
		where $\tilde{\psi}(m) = f\left( \psi(m), T_\id \right)$ for $m\in M$, is
	an isomorphism of $\hecke$-modules.
\end{proposition}
\begin{proof}
	Given $h\in \hecke_{r,q}$, $\psi\in \Hom_{\hecke} (M,
	\hecke)$, and $m\in M$, we have 
\begin{align*}
	\widetilde{h\psi}(m) &= f\left( (h\psi)(m), T_\id \right) =
	f(h\psi(m), T_\id) = f(h, \psi(m)T_\id)\\& = f(h, \psi(m)) = f(\psi(m), h) =
	f(T_\id, \psi(m)h) = f(T_\id, \psi(mh))\\& = f(\psi(mh), T_\id)  
	=\tilde{\psi}(mh) = (h\tilde{\psi})(m). 
\end{align*}
\end{proof}
\begin{proposition}
	\label{second2}
	Let $M$ and $N$ be right $\hecke$-modules. Then the following
	diagram is commutative
	\begin{equation*}
		\xymatrix@C3.5em{
		\Hom_{\hecke} (M, \hecke) \otimes
		\Hom_{\hecke}(N,M) 
		\ar[r]^{\mathfrak{F}_M\otimes \mathrm{id}} 
		\ar[d]_{\circ} & 
		\Hom_{R}(M,R)\otimes \Hom_{\hecke} (N,M) \ar[d]^{\circ}\\
		\Hom_{\hecke} (N, \hecke) \ar[r]^{\mathfrak{F}_N} & \Hom_R(N,R).
}
	\end{equation*}
	\begin{proof}
		Let $\varphi\colon N \to M$, $\psi\colon M\to \hecke$ be
		 homomorphisms of right $\hecke$-modules. Then for all
		$n\in N$
\begin{align*}
	\widetilde{\psi\varphi}(n) &= f(\psi\varphi(n), T_\id) =
	f(\psi(\varphi(n)), T_\id) = \tilde{\psi}(\phi(n)) =
	\tilde{\psi}\phi(n).
\end{align*}
Thus $\widetilde{\psi\varphi} = \tilde{\psi}\varphi$. 
	\end{proof}
\end{proposition}
Returning to our setting we will abbreviate $\mathfrak{F}_{M^\mu}$ to $\mathfrak{F}_\mu$. 
For each $k\ge 0$ define the map $\tau_k\colon F(B_{\lambda})_k \to \bmcom_{k,\lambda}$
to be the direct sum
\begin{equation*}
\tau_k := 	\bigoplus_{
\begin{smallmatrix}
	 \mu^{(1)}\gtdom \dots \gtdom \mu^{(k)}\gtdom \lambda\\
 \mu^{(1)}, \dots , \mu^{(k)}\in \Lambda(n,r)
\end{smallmatrix}}
\mathfrak{F}_{\mu^{(1)}} \otimes \id \otimes \dots \otimes \id.
\end{equation*}
Then $\tau_k$ is an isomorphism of
$\hecke$-modules for $k\ge 0$. From
Proposition~\ref{second2}, we get that for every $k\ge 1$ 
\begin{equation*}
 d_{k,0}\tau_k = \tau_{k-1} d_{k,0}.
\end{equation*}
Moreover it is obvious that  for all $k\ge 1$ and $1\le t\le k$
\begin{equation*}
 d_{k,t}\tau_k = \tau_{k-1} d_{k,t}.
\end{equation*}
Thus for all $k\ge 1$ we have $d_k \tau_k = \tau_{k-1} d_k$. This
shows that $\tau = (\tau_k)_{k\ge 1}$ is a chain transformation between
the truncated complexes $F(B_{\lambda})_{\ge 0}$ and
$\widetilde\bmcom_{\ge 0,\lambda}$. Since every $\tau_k$ is an
isomorphism of $\hecke$-modules, we get that $F(B_{\lambda})$
and $\widetilde\bmcom_{*,\lambda}$ are isomorphic in non-negative degrees
as promised.
The existence of an isomorphism in degree $-1$ follows from the commutative diagram
with  exact rows
\begin{equation*}
 \xymatrix{
 F(B_{\lambda})_1 \ar[r]^{d_{1}} \ar[d]_{\tau_1} & F(B_{\lambda})_0
 \ar[r]^{d_0} \ar[d]_{\tau_0} & F(W_\lambda)\ar[r]
 \ar@{-->}[d]^{\exists!}& 0 \\
 \bmcom_{1}^\lambda \ar[r]^{d_1} & \bmcom_0^\lambda \ar[r]^{d_0} &
 S^\lambda \ar[r] & 0.
 }
\end{equation*}
Thus we proved
\begin{theorem}
	\label{main:second_part}
	Let $\lambda$ be a partition of $r$. 
	 Then the complexes $\widetilde{\bmcom}^\lambda_*$  and $F\left( B_\lambda
	 \right)$ are isomorphic. In particular, $\widetilde\bmcom^\lambda_*$ is
	 an exact complex. 
\end{theorem}

\bibliographystyle{amsalpha}
\bibliography{q}

\end{document}